\documentclass[11pt,reqno]{amsart}
\usepackage[utf8]{inputenc}
\usepackage{geometry}

\usepackage{amsmath,amsthm,amssymb,amsfonts,color,mathtools,epsfig}
\usepackage{algorithmic,cancel,graphics}
\usepackage{algorithm,booktabs}
\usepackage[mathscr]{eucal}
\usepackage[normalem]{ulem} 
\usepackage{comment}
\usepackage{todonotes}
\usepackage{hyperref}
\usepackage{algorithmic}

\numberwithin{equation}{section} 
\newcounter{cont}[section] 

\newtheorem{theorem}[cont]{Theorem}
\newtheorem{prop}[cont]{Proposition}
\newtheorem{lemma}[cont]{Lemma}
\newtheorem{corollary}[cont]{Corollary}
\theoremstyle{definition}
\newtheorem{definition}[cont]{Definition}
 \theoremstyle{remark}
 \newtheorem{remark}[cont]{Remark}

\newcommand\e{\varepsilon}

\title[Convection-reaction equations with nonlinear diffusion]{Stability properties of solutions to convection-reaction equations\\ with nonlinear diffusion}

\author[A. Alla]{Alessandro Alla}
\address[Alessandro Alla]{Dipartimento di Matematica Guido Castelnuovo, Sapienza Universit\`a di Roma,
Piazzale Aldo Moro 5, 00185 Roma (Italy)}
\email{alessandro.alla@uniroma1.it}

\author[A. De Luca]{Alessandra De Luca}
\address[Alessandra De Luca]{Dipartimento di Matematica ``Giuseppe Peano'', Universit\`a di Torino,
Via Verdi 8, 10124 Torino (Italy)}
\email{a.deluca@unito.it}

\author[R. Folino]{Raffaele Folino}
\address[Raffaele Folino]{Departamento de Matem\'aticas y Mec\'anica, Instituto de Investigaciones en Matem\'aticas Aplicadas y en Sistemas, Universidad Nacional Aut\'onoma de M\'exico, Circuito Escolar s/n, C.P. 04510 Cd. de M\'exico (M\'exico)}
\email{folino@aries.iimas.unam.mx}
 
\author[M. Strani]{Marta Strani}
\address[Marta Strani]{Dipartimento di Scienze Molecolari e Nanosistemi, Universit\`a Ca' Foscari Venezia Mestre,
Campus Scientifico, Via Torino 155, 30170 Venezia Mestre (Italy)}
\email{marta.strani@unive.it}

\date{}

\begin{document}

\maketitle
\begin{abstract}
In this paper we study a convection-reaction-diffusion equation of the form
\begin{equation*}
u_t=\varepsilon(h(u)u_x)_x-f(u)_x+f'(u), \quad t>0,
\end{equation*}
with a nonlinear diffusion in a bounded interval of the real line. In particular, we first focus our attention on the existence of stationary solutions with at most one zero inside the interval, studying their behavior with respect to the viscosity coefficient $\varepsilon>0$ and their stability/instability properties. Then, we investigate the large time behavior of the solutions for finite times and the asymptotic regime. We  also show numerically that, for a particular class of initial data, the so-called metastable behavior occurs, meaning that the time-dependent solution persists for an exponentially long (with respect to $\varepsilon$) time in a transition non-stable phase, before converging to a stable configuration. 
\end{abstract}

\section{Introduction and main results}
We study the following initial boundary value problem
\begin{equation}\label{prob}
	\begin{cases} 
		u_t=\varepsilon(h(u)u_x)_x -f(u)_x + f'(u), \qquad    & x \in (0,\ell), \quad t>0, \\
		u(x,0) = u_0(x), &x\in[0,\ell],\\
		u(0,t) = u(\ell,t)=0,  & t\geq0,
	\end{cases}	
\end{equation}
where $u:=u(x,t)$ is the unknown function and $\e>0$ is a fixed parameter to be taken conveniently small which, 
compared to the fluid-dynamics phenomena, can be interpreted as a viscosity coefficient.
As concerning the functions appearing in the convection-reaction-diffusion equation, the diffusion term is modeled by a $C^1$ function $h$ such that 
\begin{equation}\label{segnodih}
	h(u)\geq h_0 >0\quad \text{ and }\quad h'(u)u\leq0, \qquad \mbox{ for any } u\in[-R,R],
\end{equation}
for some positive constants $h_0$ and $R$.
The positivity of the function $h$ ensures that \eqref{prob} is an IBVP for a parabolic PDE when $u\in[-R,R]$;
in all the paper we shall consider initial data $u_0$ such that the solution of \eqref{prob} belongs to $[-R,R]$ for any $t\geq0$.
There are several physical situations where the mathematical modelling involves a (strictly positive) phase-dependent diffusion coefficient,
such as the Mullins diffusion model for thermal grooving \cite{Broa,Mullins}, where the function 
\begin{equation}\label{eq:Mullins}
	h(u)= (1+u^2)^{-1}
\end{equation}
is used to describe the development of surface groove profiles on a heated polycrystal by the mechanism of evaporation-condensation.
Another example of function satisfying assumption \eqref{segnodih} is the Gaussian $h(u)=e^{-u^2}$; 
notice that, for any $R>0$ it is possible to choice $h_0>0$ (depending on $R$) such that assumption \eqref{segnodih} is satisfied
in both the case of the Mullins diffusion \eqref{eq:Mullins} and that one of the Gaussian function.

Finally, the flux $f$ is a $C^2$ function satisfying the following conditions 
\begin{equation}\label{segnodif}
	f(0)=f'(0)=0,\quad f''(u)>0\quad \text{and} \quad f'(-u)=-f'(u), \qquad \mbox{for any } u\in[-R,R].
\end{equation}
The previous assumptions imply that
\begin{equation}\label{segnodif1}
	f'(u)u>0, \qquad \mbox{ for any } u \neq 0,
\end{equation}
and that for all $A\in(0,R)$ there exists a constant $K$ (depending on $A$) such that
\begin{equation}\label{segnodif2}
 f'(u) \leq K u, \qquad \mbox{ for any }  u \in [0,A].
\end{equation}
The convection-reaction-diffusion equation in \eqref{prob} is a generalization of the so called Burgers--Sivashinsky equation
\begin{equation}\label{BSeq}
	u_t=\varepsilon u_{xx} -u u_x+u,
\end{equation}
which can be obtained by choosing $h(u)=1$ and $f(u)=u^2/2$, see \cite{BKS,marta, sunward} and references therein. 
Equation \eqref{BSeq} arises from the study of the dynamics of an upwardly propagating flame-front in a vertical channel  (see \cite{instabilities} for the derivation of the model). Indeed (and differently from the mathematical model in an unbounded domain, which is a typical example of a free interface system), it comes out in the study of premixed gas flames in vertical tubes, where the upwardly propagating flames often assume a characteristic shape with the tip of their parabolic profile located somewhere near the center of the channel. 

Precisely, if one considers the nonlinear evolution equation for  the flame interface  derived in \cite{instabilities} under various physical assumptions, the dimensionless shape of the flame front  is modelized by a function $y=y(x,t)$ which satisfies
$$y_t-\frac{1}{2}y^2_x= \e y_{xx}+y- \int_0^{\ell} y \, dx, \qquad y(0,t)=y(\ell,t)=0.$$
Hence, the function $u=u(x,t)$ defined as the slope $u=-y_x$ satisfies \eqref{BSeq}; we remark that if $u(\cdot, t)$ is close to be linear, then $y$ is close to a parabola, and its tip  corresponds with the point where $u$ vanishes, called {\it interface}.

As concerning the dynamics of the flame front  $y$, it is worth mentioning the {work \cite{instabilities2}}, where numerical simulations and rigorous results show that  the flame motion markedly depends on the initial conditions and  displays the following trend for $\e \ll 1$:
\begin{itemize}
\item at the early stage of its developments, the solution is rapidly attracted to some intermediate stage where the parabolic flame front interface is formed and  assumes an asymmetric parabolic profile; this is however only a transient phase, but the subsequent motion occurs at a very low speed rate, which becomes even slower provided $\e$ to be small enough. In this process, the tip of the parabola gradually moves towards one of the walls with an exponentially slow speed; the fact that the time interval corresponding to this phase may be extremely long for small $\e$ creates the illusion that the flame front has reached some final equilibrium. However,  this is merely  a {\it quasi-steady} state, and, after a very long time,
 the flame front finally collapses by reaching one of the walls and attains its final equilibrium state. 
\end{itemize}
Translated into the dynamics of the solution $u$ to \eqref{BSeq}, and recalling that the tip of the parabola corresponds to the interface $u=0$, we observe
a {\it slow motion} phenomenon, which is very well known in literature as metastability; precisely, as rigorously studied in \cite{BKS} when $f(u)=u^2/2$ and \cite{marta, sunward} in the case of a generic flux function $f$, a necessary condition for the solutions to \eqref{BSeq} to exhibit a metastable behavior is that the initial datum satisfies
\begin{equation}\label{NC}
	u_0(x) <0, \quad \mbox{for} \quad x \in (0,a) \qquad \mbox{and} \qquad u_0(x) >0, \quad \mbox{for} \quad x \in (a, \ell),
\end{equation} 
for some $a \in (0,\ell)$. Under assumption \eqref{NC}, the dynamics for \eqref{BSeq} is the following:
\begin{itemize}
\item in a first transient phase the initial configuration develops into a layered profile, which is far from any stable configuration of the system, but the subsequent motion of the solution, studied as the one-dimensional motion of the interface $I := \{ a(t) \, : \, u(t, a(t)) =0 \}$ towards one of the walls $x=0$ or $x=\ell$, is extremely slow, provided $\e\ll 1$. 
\end{itemize}

The numerical computations contained in \cite{sunward} also suggest that other kind of initial data lead to solutions that attain a stable equilibrium configuration in an $\mathcal{O}(1)$ time interval.

There are several contributions who investigated problems like \eqref{BSeq} in different dimensional sets and under more general forms for the nonlinearity appearing in the equation: to name some of these papers, we recall here \cite{BOR,CLS, H, HS, HW}.
On the other hand, the problem with a nonlinear diffusion seems to be unexplored until now, despite its  several applications. Also, apart from the context of gas combustion, it is worth mentioning that equation \eqref{BSeq} also arises in the statistical description of biological evolution, where $-y$ (recall again  that $u=-y_x$) represents the entropy, see \cite{9inBKS}.

\subsection{Main results} In this section we give a description of the main results contained in this paper, which concern both the initial boundary value problem \eqref{prob}
and the stationary problem associated to it. For the sake of simplicity, we will consider initial data that change sign at most one time inside the interval $(0,\ell)$; in particular, given the parabolic nature of \eqref{prob}, if we assume $u_0 \in C^0([0,\ell])$ to be such that $u_0(0)=u_0(\ell)=0$, there exists a unique classical solution to \eqref{prob} (see for instance \cite{Lady}).

The first part of our results, contained in Sections \ref{sec:stat} and \ref{sec3}, concerns the stationary problems associated to \eqref{prob}, 
i.e. the solutions $u:=u(x)$ to the following BVP
\begin{equation}\label{stat:intro}
\begin{cases} \varepsilon(h(u)u_x)_x -f(u)_x + f'(u)=0, \qquad  x \in (0,\ell), \\
u(0) = u(\ell)=0. 
\end{cases}
\end{equation}
At first, we give a complete description of the solutions to \eqref{stat:intro}, by proving that, if $\e< \gamma \, (\ell/ \pi)^{2}$ (where $\gamma$ is a positive constant depending on $f$ and $h$), then there exist a positive solution $u_{+,\e}$ and a negative solution $u_{-,\e}$ to \eqref{stat:intro}. Likewise, if $\e< \gamma \, (\ell/ 2\pi)^{2}$, there exist two solutions $u_{ 1^{\pm},\e}$ with exactly one zero in the interval $(0,\ell)$; for the precise statements, see Propositions \ref{positivesolut}, \ref{esistelau-} and \ref{onezerosol}.
The previous results can be generalized to the case of solutions with more than one zero in $(0,\ell)$, see Remark \ref{morezeros}.

Going further, in Section \ref{sec3} we firstly describe the behavior of the steady states with respect to $\e>0$; precisely, in Lemmas \ref{lemmadecrescenza} and \ref{lemmacrescenza} we prove that the positive and negative steady states $u_{\pm,\e}$ are respectively decreasing and increasing with respect to the parameter $\e$. The latter property can be used to prove the following convergence-type results for $\e \to 0^+$ (see Propositions \ref{convergenzau+} and \ref{velocitaconv} and Corollary \ref{convuM}).

\begin{itemize}
\item The positive solution $u_{+,\e}$ converges to $x$ as $\e \to 0^+$, while the negative solution $u_{-,\e}$ converges to $x-\ell$ as $\e \to 0^+$, uniformly on compact subsets of $[0,\ell)$. Moreover, 
it holds that
$$| u'_{+,\e}(\ell) | = \mathcal{O} \left( \frac{1}{\e}\right) \qquad \mbox{and} \qquad | u'_{-,\e}(0) | = \mathcal{O} \left( \frac{1}{\e}\right). $$

\item The function $u_{1^+,\e}$ converges uniformly on compact subsets of $\left[0,\frac{\ell}{2}\right) \cup \left(\frac{\ell}{2}, \ell\right)$ to the function  $x \, \chi_{(0,\ell/2)} + (x-\ell) \, \chi_{(\ell/2,\ell)}$, while the function $u_{1^-,\e}$ converges uniformly on compact subsets of $(0,\ell)$ to to the function $x-\frac{\ell}{2}.$
\end{itemize}
Finally, in Theorem \ref{stabthm} we describe the stability properties of the steady states; we show that, while $u_{\pm,\e}$ are asymptotically stable, 
as for the solutions $u_{1^{\pm}, \e}$, one can choose initial data arbitrarily close to them for which the corresponding time dependent solution converges to either 
$u_{+,\e}$ or $u_{-,\e}$. In particular, the most interesting case is the one of $u_{1^-,\e}$, which is not only unstable, but {\it metastable}: specifically, starting from small perturbations of such an unstable steady state, the corresponding time-dependent solutions are pushed away towards one of the stable equilibrium configurations $u_{\pm,\e}$, but the motion can be exponentially slow with respect to the parameter $\e$. 
We refer the reader to \cite{BKS,marta,sunward} for the linear diffusion case $h\equiv1$.
This phenomenon will be the object of the last part of the paper; to start with, in the first part of Section \ref{hyperbolic} we consider the first-order hyperbolic equation
 \begin{equation*}\label{probhyp}
\begin{cases} U_t+f'(U)U_x -f'(U)=0,   \qquad & x \in (0,\ell), \\
U(0,t) = U(\ell,t)=0,  & t>0, \\
U(x,0) = u_0(x), & x\in[0,\ell],\\
\end{cases}	
\end{equation*}
formally obtained by setting $\e=0$ in \eqref{prob}. By considering viscosity solutions as they are classically defined, we restrict the analysis to the following three  kind of initial data:
\begin{itemize}
\item $u_0$ of type A, meaning $u_0 >0$ (or, equivalently $u_0<0$) in $(0,\ell)$,
\item $u_0$ of type B, meaning $u_0 <0$ in $(0,x_0)$ and $u_0>0$ in $(x_0,\ell)$, for some $x_0 \in (0,\ell)$,
\item $u_0$ of type C, meaning $u_0 >0$ in $(0,x_0)$ and $u_0<0$ in $(x_0,\ell)$, for some $x_0 \in (0,\ell)$,
\end{itemize}
and we prove the following result (see Theorem \ref{th:asyhyp} for the precise statement):
\begin{itemize}
\item If $u_0$ is of type $A$, then $U(x,t)$ converges to either $x$ or $x-\ell$ as $t \to +\infty$, depending on whether $u_0>0$ or $u_0<0$. 
\item If $u_0$ is of type $B$, then $U(x,t)$ converges to  $x-x_0$ as $t \to +\infty$.
\item If $u_0$ is of type $C$, then $U(x,t)$ may converge to either $x$,  $x-\ell$, or $x \, \chi_{(0,\ell/2)} + (x-\ell) \, \chi_{(\ell/2,\ell)}$ as $t \to +\infty$ (depending on whether $x_0> \frac{\ell}{2}$, $x_0=\frac{\ell}{2}$ or $x_0 < \frac{\ell}{2}$).
\end{itemize}
The previous description allows us to understand the finite time behavior of the parabolic problem \eqref{prob} when $\e$ is positive but small. This results are contained in Theorem \ref{teo:finitetime} and the subsequent remark. We underline that initial data of type B are the most interesting here, since (as in the classical linear case) they lead to a metastable behavior for the time-dependent solutions to \eqref{prob} (recall the necessary condition given in \eqref{NC}). Having in mind  to illustrate this peculiar phenomenon,  the last  section of the paper is dedicated to some numerical simulations showing the long time dynamics of the solutions to \eqref{prob}. In particular, in Section \ref{numerics} we numerically show the stable/unstable nature of the four different steady states, paying a particular attention to initial data of type B, for which one can see the (exponentially) slow motion of the solutions 
when $\e>0$ is very small.

\section{The stationary problem}\label{sec:stat}
In this section we perform the study of the  stationary problem related to \eqref{prob}, namely
\begin{equation}\label{stationaryprob}
\begin{cases} \varepsilon(h(u)u_x)_x =f(u)_x - f'(u),  \qquad x\in (0,\ell), \\
 u(0)=u(\ell)=0.
\end{cases}	
\end{equation}
In particular we focus our attention on solutions with at most one zero inside the interval $(0,\ell)$, and we show that there exist exactly four solutions to \eqref{stationaryprob} enjoying the following property: two of them do not have zeroes inside the interval (one is strictly positive and the other one is strictly negative, see Propositions \ref{positivesolut} and \ref{esistelau-}), while the other two have  exactly one zero, as shown in Proposition \ref{onezerosol}.
We begin our analysis by proving a necessary condition for the existence of nontrivial solutions to \eqref{stationaryprob}.

\begin{prop}\label{propesistenza}
There exists $\gamma >0$ such that, if $\varepsilon \geq \gamma \, \ell^2 \, \pi^{-2}$, then there are no solutions to \eqref{stationaryprob} apart from the trivial one $u(x)=0$ for all $x \in [0,\ell].$
\end{prop}
\begin{proof}
Let us define 
$$m:= \min_{x \in [0,\ell]} h(u(x)),$$ 
then multiply \eqref{stationaryprob} by $u$ and integrate by parts to obtain
$$ \varepsilon \int_0^{\ell} h(u)u^2_x \, dx =  \int_0^{\ell} f(u) u_x \, dx +\int_0^{\ell} f'(u)u \, dx.  $$
Denoting with $F$ the primitive of the function $f$, since
$$\int_0^{\ell} f(u) u_x \, dx = F(u(\ell))-F(u(0))=F(0)-F(0)=0,$$
we have
$$m \, \varepsilon \int_0^{\ell} u_x^2 \, dx   \leq \frac{ K \ell^2}{\pi^2} \int_0^{\ell} u_x^2 \, dx, $$
where in the right-hand side we used \eqref{segnodif2} and the Poincar\`e-Sobolev inequality. The thesis then follows with $\gamma=Km^{-1}.$
\end{proof}

We go on giving the notions of sub/super-solution to \eqref{stationaryprob} of class $C^2$ which can be used as tools to prove the existence of a solution to \eqref{stationaryprob} by a standard monotone iteration technique (see \cite{evans}). In addition, even if throughout this section we will always deal with classical sub/super-solutions, for the sake of completeness, we provide the reader with the rigorous definition of \emph{weak} sub/super-solution to \eqref{stationaryprob}: this notion will come into play in Section \ref{sec3} in order to construct certain sub/super-solutions with discontinuities in the first derivative, when proving the stability/instability of the steady states introduced in the present section.
\begin{definition}\label{def_subsuper}
A function $u \in H^1([0,\ell])$ is a weak sub-solution (respectively a weak super-solution) to \eqref{stationaryprob} if $u(0), u(\ell) \leq 0$ (respectively $u(0), u(\ell) \geq 0$) and
\begin{equation*}\label{dis}
\int_0^\ell \left\{\varepsilon h(u)u_x \varphi_x+[f'(u)u_x-f'(u)]\varphi \right\} \, dx \leq 0 \qquad \rm{(respectively \geq 0 )}
\end{equation*}
for all $\varphi \in C^1([0,\ell])$, $\varphi \geq 0$ in $[0,\ell]$ such that $\varphi(0)=\varphi(\ell)=0$.

A function $u \in C^2([0,\ell])$ which satisfies \eqref{dis} is a sub-solution (respectively a  super-solution)  to \eqref{stationaryprob}. In particular, for all $x \in [0,\ell]$  it satisfies 
\begin{equation*}
-\varepsilon(h(u)u_x)_x  +f'(u)u_x-f'(u)\leq 0\qquad \rm{(respectively \geq 0 )}.
\end{equation*}
\end{definition}


\begin{prop}\label{positivesolut}
For any $0<\e< \gamma \, (\ell/ \pi)^{2}$, with $\gamma$ being as in Proposition \ref{propesistenza}, problem \eqref{stationaryprob} admits a  positive solution $u_{+,\varepsilon}$ such that 
\begin{equation}\label{propr-sol-pos}
     u_{+,\varepsilon}(x) \leq x, \quad u_{+,\varepsilon}'(x)\leq 1 \qquad \mbox{and} \qquad   u_{+,\varepsilon}''(x)\leq 0,
\end{equation}
for all $x \in (0,\ell).$
\end{prop}

\begin{proof}
In order to prove the existence of a classical solution to problem \eqref{stationaryprob}, we seek a smooth super-solution $v_1$ and a smooth sub-solution $v_2$ such that $v_1\geq v_2$, according to the standard monotone methods \cite{evans,sattinger}.
To this purpose, let 
\begin{equation}\label{operatoreL}
	Lu:= -\varepsilon(h(u)u_x)_x  +f'(u)u_x-f'(u).  
\end{equation}
We immediately observe that the function $v_1(x)=x$ is trivially a super-solution: indeed $Lv_1(x)= -\varepsilon h'(x) \geq 0$ in virtue of assumption \eqref{segnodih}, which implies that $h'(x)\leq 0 $ (since $x \in [0,\ell]$).
Going on, we consider the function $v_2(x)=\alpha\sin\left(\frac{\pi}{\ell}x\right)$
with $\alpha>0$ (to be determined later) such that
\begin{equation}\label{sceltadialpha}
    \alpha\frac{\pi}{\ell}\leq 1. 
\end{equation}
By direct calculations, we have that 
\begin{equation}\label{Lv2<0}
\begin{split}
    Lv_2=\, &-\varepsilon \alpha^2\left(\frac{\pi}{\ell}\right) ^2 h'(v_2)\cos^2\left(\frac{\pi}{\ell}x\right)+\varepsilon \alpha\left(\frac{\pi}{\ell}\right) ^2 \sin\left(\frac{\pi}{\ell}x\right)h(v_2)\\
    &+ f'(v_2) \left[\alpha\frac{\pi}{\ell}\cos\left(\frac{\pi}{\ell}x\right)-1\right]\\
    \leq & -\varepsilon \alpha\left(\frac{\pi}{\ell}\right) h'(v_2)+\varepsilon \alpha\left(\frac{\pi}{\ell}\right) ^2 h(v_2)+f'(v_2) \left[\alpha\frac{\pi}{\ell}-1\right].\\
\end{split}
\end{equation}
We notice that in the last row of \eqref{Lv2<0}, the first two terms are positive by assumption \eqref{segnodih} and the term $f'(v_2) \left[\alpha\frac{\pi}{\ell}-1\right]$ is negative in virtue of \eqref{segnodif1} and \eqref{sceltadialpha}. Hence, defining 
$$m_1:=\max_{x\in [0,\ell]}|h'(v_2(x))|,\quad m_2:=\max_{x\in [0,\ell]} h(v_2(x)),$$ and $$m_3:= \max_{x\in [0,\ell]} f'(v_2(x)),$$ we can conclude that there exists a suitable $\alpha>0$ depending on $\e,\ell, m_1,m_2,m_3$ such that $Lv_2\leq 0$. Thus $v_2$ is a sub-solution, and $v_1\geq v_2$ thanks to \eqref{sceltadialpha}, as desired. 
From this, we can deduce that there exists a stationary solution $u_{+,\varepsilon}$ to problem \eqref{prob} such that $$v_2(x)\leq u_{+,\varepsilon}(x)\leq v_1(x).$$
Therefore $u_{+,\varepsilon}$ is positive and such that $u_{+,\varepsilon} (x) \leq x$ for any $x \in (0,\ell)$.
In particular,
$u_{+,\varepsilon}$ has at least one maximum. In order to show that $u''_{+,\varepsilon}\leq 0$, we prove that such maximum is unique, arguing by contradiction: we assume that $u_{+,\varepsilon}$ has a local minimum at some point $x_2\in (0,\ell)$. In this case, by \eqref{stationaryprob} we would have that 
\begin{equation*}\label{eq:exmin}
  \varepsilon h(u_{+,\varepsilon}(x_2))u_{+,\varepsilon}''(x_2)= -f'(u_{+,\varepsilon}(x_2)),
  \end{equation*}
which in turn gives rise to a contradiction since the left hand side is positive whereas the right hand side is negative. Hence $u''_{+,\varepsilon}$ is negative for all  $x \in (0,\ell)$.
Finally, since $u'_{+,\varepsilon}$ is a decreasing function and $u_{+,\varepsilon}'(0)\leq 1$ (since $u_{+,\varepsilon} \leq x$), it easily follows that $u_{+,\varepsilon}'(x)\leq 1$ for all $x \in (0,\ell)$. 
\end{proof}

\subsection*{Uniqueness of the solution} As concerning the uniqueness of the positive solution of \eqref{stationaryprob}, by proceeding as in \cite{BKS} we consider  the following Cauchy problem
\begin{equation}\label{cauchyprob}
\begin{cases} \varepsilon(h(u)u_x)_x =f(u)_x - f'(u),  \qquad x>0, \\
 u(0)=0, \quad u_x(0)=\alpha \in (0,\bar \alpha),
\end{cases}	
\end{equation}
with $\bar \alpha>0$ to be properly chosen; problem \eqref{cauchyprob} can be rewritten as the first order system
\begin{equation*}
	\begin{cases}
		u_x=v,\\
		\e h(u)v_x=-{\e h'(u)}v^2+{f'(u)}v-{f'(u)}, 
	\end{cases}
\end{equation*}
with initial conditions  $u(0)=0$ and $v(0)=\alpha$.
Notice that, since $h'(0)=f'(0)=0$, then $v_x(0)=0$ and, in order to establish the convexity of the function $u$ near $0$, we need to compute the sign of $v_{xx}(0)$.
A direct computation gives
$$v_{xx}(0)=-\frac{h''(0)}{h(0)}\alpha^3+\frac{f''(0)}{\e h(0)}\alpha(\alpha-1)=\frac{\alpha}{\e h(0)}\left[-\e h''(0)\alpha^2+f''(0)(\alpha-1)\right],$$
and as a consequence, $v_{xx}(0)<0$ if and only if $\alpha\in(0,\bar\alpha)$, with
\begin{equation}\label{baralfa}
	\bar\alpha:=\frac{2}{1+\sqrt{1-4\e\frac{h''(0)}{f''(0)}}}.
\end{equation}
Notice that $\bar\alpha\to1$ as $\e\to0^+$ and that in the constant case $h\equiv 1$, one has $\bar\alpha=1$. 
It is easy to check that if $\alpha\geq\bar\alpha$, the corresponding solution of \eqref{cauchyprob} is monotone increasing, while
 if we choose $\alpha \in (0,\bar \alpha)$, then there exists $L:=L(\alpha)>0$ such that $u>0$ in $(0,L(\alpha))$ and $u(L(\alpha))=0$. Hence, since we already proved that there exists $\alpha^* \in (0,\bar \alpha)$ such that $L(\alpha^*)=\ell$ (corresponding to the positive solution $u_{+,\e}$), in order to prove the uniqueness of the positive solution $u_{+,\e}$ it is sufficient to prove that $L$ is a strictly increasing function of $\alpha$. The monotonicity of $L$ has been proved for the constant case $h\equiv1$ and for the reaction term $f(u)=u^2/2$ in \cite{BKS}, where the authors are able to provide an expression for $L(\alpha)$ (see \cite[Proposition 4.2, Appendix A]{BKS}).
Following their strategy, we introduce the change of variable $s=u(x)$ and $p=p(s)=h(u(x(s)))u_x(x(s))$; since $p_s= (h u_x)_x \frac{d x}{ds}= (h u_x)_x \,  \frac{h}{p}$, the function $p$ solves
$$ \e p  \, p_s =f'(s)(p-h(s)).$$
Conversely to the constant case $h\equiv1$, in the latter equations we can not separate variables 
and we can no longer obtain an expression for $L(\alpha)$, as done in \cite{BKS}. 

In general, when $h$ is not constant, providing an expression for $L(\alpha)$ and an analytical proof of its monotonic behavior is not an easy task and it still remains an open problem; however, we can provide numerical evidences when $h$ is explicitly given, for example by the Mullins diffusion \eqref{eq:Mullins} or by the Gaussian function (see Figure \ref{lal}). 
\begin{figure}[htbp]
\includegraphics[scale=0.4]{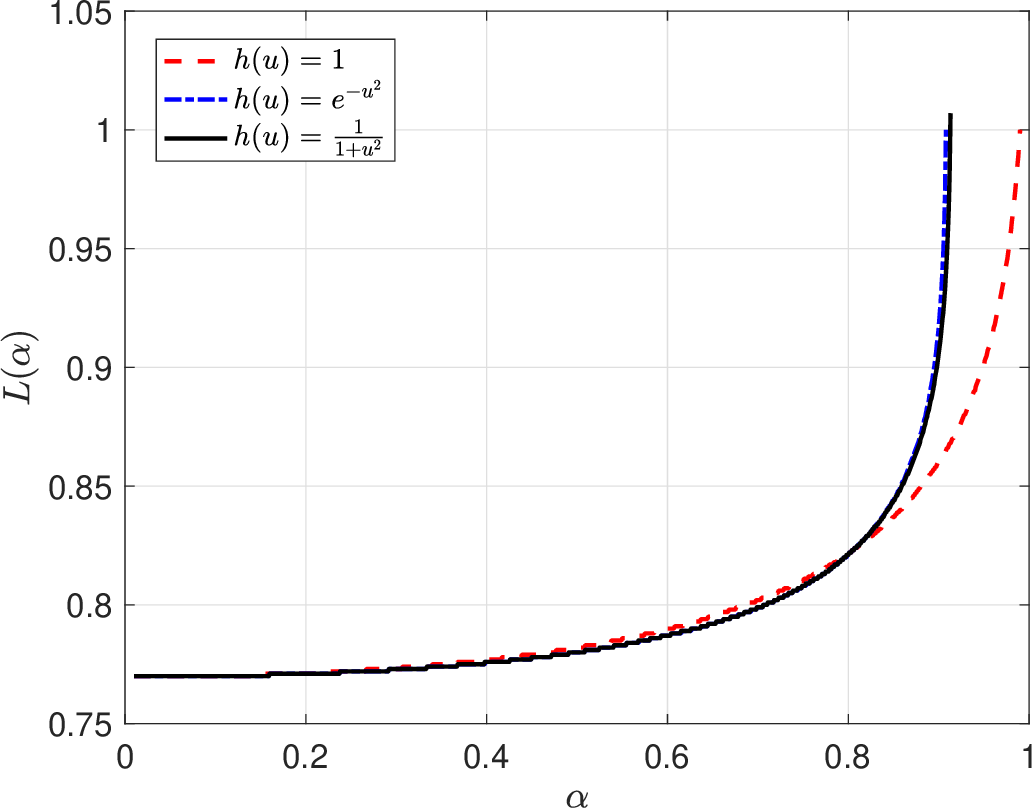}
\includegraphics[scale=0.4]{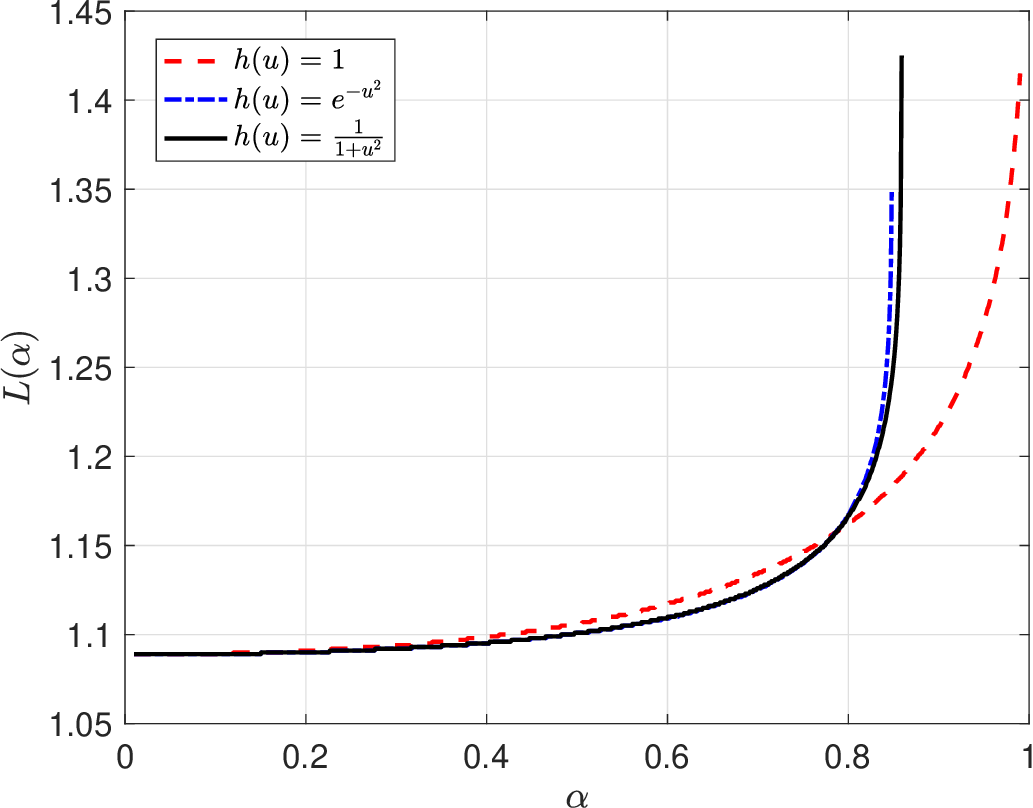}
\caption{Behavior of $L(\alpha)$ with three different diffusions and $\varepsilon = 0.06$ (left) and $\varepsilon =0.12$ (right).}
\label{lal}
\end{figure}

{As one can see from Figure \ref{lal} the function $L(\alpha)$ is strictly increasing. 
The computation of $L(\alpha)$ in the case of a nonlinear diffusion and for $\e=0.06$ (left figure) has been stopped for $\alpha\approx0.9$; 
in this case, $h''(0)=-2$ (both for the Mullins and the Gaussian) and $f''(0)=1$, and \eqref{baralfa} gives $\bar\alpha=0.90230211$.
On the right figure we choose $\e=0.12$, so that $\bar \alpha$ is smaller; if we compute \eqref{baralfa}, in this case we obtain $\bar \alpha = 0,8\bar 3$, as shown in figure \ref{lal}.
We finally stress that $L(\alpha) \to +\infty$ as $\alpha \to \bar \alpha$.  
The numerical simulations have been carried out using the Matlab function {\tt ode23tb}.}

%
Motivated by the numerical simulations, we are confident that the monotonic property of the function $L$ is preserved also in the case of a nonlinear diffusion as the one considered in \eqref{segnodih}, even if we are not able to prove it rigorously at the present time. Hence, from now on, we will assume that the positive solution to \eqref{stationaryprob} is unique.
\vskip0.5cm
In a completely analogous way, we can prove that there exists a  unique negative solution to problem \eqref{stationaryprob}.
\begin{prop}\label{esistelau-}
For any $0<\e< \gamma \, (\ell/ \pi)^{2}$, with $\gamma$ being as in Proposition \ref{propesistenza}, 
problem \eqref{stationaryprob} admits a negative solution $u_{-,\varepsilon}$ such that 
\begin{equation*}
     u_{-,\varepsilon}(x) \geq x-\ell, \quad u_{-,\varepsilon}'(x)\leq 1 \qquad \mbox{and} \qquad   u_{-,\varepsilon}''(x)\geq 0,
\end{equation*}
for all $x \in (0,\ell).$
\end{prop}
\begin{proof}
We give only a hint of the proof, since one can proceed exactly as in the proof of Proposition \ref{positivesolut}. In particular, it is sufficient to consider $v_1(x)=x-\ell$ as a sub-solution (indeed, in light of assumption \eqref{segnodih}, $h'(x-\ell)\geq0$ since $x\leq\ell$), and $v_2(x)=-\alpha\sin\left(\frac{\pi}{\ell}x\right) $ as a super-solution, for some $\alpha>0$ satisfying \eqref{sceltadialpha}. If we  now compute  $L v_2$, with $L$ defined in \eqref{operatoreL}, we have 
\begin{equation*}
\begin{split}
Lv_2=& -\e\alpha^2\left(\frac{\pi}{\ell}\right)^2h'(v_2)\cos^2\left(\frac{\pi}{\ell}x\right)-\e\alpha\left(\frac{\pi}{\ell}\right) h(v_2) \sin\left(\frac{\pi}{\ell}x\right)\\
& - f'(v_2) \left[\alpha\frac{\pi}{\ell}\cos\left(\frac{\pi}{\ell}x\right)+1\right] \\
\geq & -\e\alpha^2\left(\frac{\pi}{\ell}\right)^2h'(v_2)-\e\alpha\left(\frac{\pi}{\ell}\right) h(v_2)- f'(v_2) \left[-\alpha\frac{\pi}{\ell}+1\right]. \\
\end{split}
\end{equation*}
We thus take advantage of the fact that $f'(v_2)<0$ by \eqref{segnodif1} and we use \eqref{sceltadialpha} to prove that $Lv_2\geq 0$ for some suitable $\alpha>0$. As for the property $u''_{-,\e}(x)\geq 0$, it is enough to argue as in the proof of Proposition \ref{positivesolut} by showing that $u'_{-,\e}$ does not have a local maximum; from this, exploiting that $u'_{-,\e}(\ell)\leq 1$, one can additionally deduce that $u'_{-,\e}(x)\leq 1$.
\end{proof}

\begin{remark}\label{rem1}
It is worth observing that if in addition $h$ is symmetric, i.e.  $h(u)=h(-u)$,
then by making use of the symmetries of the problem
one can easily find that 
\begin{equation}\label{u-}
   u_{-,\varepsilon}(x)  = -u_{+,\varepsilon}(\ell-x). 
\end{equation}
Indeed, the right-hand side in \eqref{u-} is trivially negative and $L(-u_{+,\e}(\ell-x))=0$ by \eqref{segnodif} and since $h'(-u)=-h'(u)$ (a direct consequence of the symmetry of $h$).
At last, by \eqref{propr-sol-pos} and direct computations, one can easily deduce that 
$$-u_{+,\varepsilon}(\ell-x)\geq x-\ell, \quad [-u_{+,\varepsilon}(\ell-\cdot)]'(x)\leq 1,\quad [-u_{+,\varepsilon}(\ell-\cdot)]''(x)\geq 0,$$
for all $x \in (0,\ell)$.
\end{remark}

For future uses, we prove the following results on the monotonic behaviour of $u_{\pm,\e}$ with respect to the length of the interval in which it is defined. For this purpose, given any  interval $(a,b) \subset (0,\ell)$ (with possibly either $a=0$ or $b=\ell$), we denote by $u_{\pm, \e} (x;a,b)$ the unique positive (negative respectively) solution to
\begin{equation} \label{probsuaeb}
\begin{cases} \varepsilon(h(u)u_x)_x =f(u)_x - f'(u) ,  &\text{$x\in (a,b)$}, \\
 u(a)=u(b)=0.
\end{cases}	
\end{equation}
We stress that solutions to \eqref{probsuaeb} can be derived starting from the solutions $u_{+,\e}$ and $u_{-,\e}$ of Proposition \ref{positivesolut} and Proposition \ref{esistelau-} respectively, by applying a suitable rescaling.
\begin{prop}\label{monotonic}
The positive solution $u_{+,\e}(x;0,b)$ and the negative one $u_{-,\e}(x;b,\ell)$ of problem \eqref{probsuaeb} are both monotonically increasing with respect to $b\in (0,\ell)$. 
\end{prop}
\begin{proof}
The function defined as
\begin{equation*}
\bar{u}(x) := \left\{ \begin{aligned} &u_{+,\e}(x;0,b), \quad &\mbox{for} \quad x \in [0,b], \\ &0,\quad \, &\mbox{for} \quad x \in [b,\ell], \\\end{aligned}\right.
\end{equation*}
turns out to be a weak sub-solution of problem \eqref{stationaryprob}: indeed, for any test function $\varphi\geq 0$ in $(0,\ell)$ such that $\varphi(0)=\varphi(\ell)=0$ it holds that 
\begin{equation*}
\int_0^\ell \{\e h(\bar{u})\bar{u}_x\varphi_x-f(\bar{u})\varphi_x-f'(\bar{u})\varphi\} dx= \varepsilon h(\bar{u}) \bar{u}'(b^-) \varphi(b) \leq 0. 
\end{equation*} 
Moreover it is easy to see that $\bar{u}(x)\leq x$. Since $x$ is a super-solution of problem \eqref{stationaryprob}, as shown in the proof of Proposition \ref{positivesolut}, we have that $\bar{u}(x)\leq u_{+,\e} (x)$ by uniqueness of the positive steady state $u_{+,\e}$; hence
\begin{equation*}
u_{+,\e}(x;0,b)\leq u_{+,\e}(x;0,\ell)\quad \text{for every $x\in (0,b)$}. 
\end{equation*}
This can be generalized to any $\bar{b}>b$ in place of $\ell$. The same monotonic property can be proved for the negative solution $u_{-,\e}(x;b,\ell)$ by reasoning as above: it is sufficient to consider the following super-solution 
\begin{equation*}
\bar{\bar{u}}(x) := \left\{ \begin{aligned} &0 \quad &\mbox{for} \quad x \in [0,b], \\ &u_{-,\e}(x;b,\ell)\quad \, &\mbox{for} \quad x \in [b,\ell], \\\end{aligned}\right.
\end{equation*}
which satisfies $\bar{\bar{u}}(x) \geq x-\ell$, since now $x-\ell$ is a sub-solution. 
\end{proof}
As a byproduct we obtain the following result  that is crucial to prove Proposition \ref{onezerosol}.
\begin{corollary}\label{risolutivo}
If $0<b-a<c-b$, then
\begin{equation*}\label{disvoluta}
\frac{d}{dx} u_{-,\e}(x;a,b)\biggl|_{x=b}< \frac{d}{dx} u_{+,\e}(x;b,c)\biggl|_{x=b}. 
\end{equation*}
\end{corollary}
\begin{proof}
From Proposition \ref{monotonic} we can immediately infer that if $b_1<b_2$, then
\begin{equation}\label{corol1}
\frac{d}{dx} u_{+,\e}(x;0,b_1)\biggl|_{x=0} < \frac{d}{dx} u_{+,\e}(x;0,b_2)\biggl|_{x=0}   
\end{equation}
and 
\begin{equation}\label{corol2}
\frac{d}{dx} u_{-,\e}(x;b_1,\ell)\biggl|_{x=\ell}> \frac{d}{dx} u_{-,\e}(x;b_2,\ell)\biggl|_{x=\ell}.
\end{equation}
Thus, after observing that 
\begin{equation*}
-u_{-,\e}(b+a-x;b,2b-a)<u_{+,\e}(x;b,c),
\end{equation*}
the thesis follows from \eqref{corol1} and \eqref{corol2}, since $2b-a<c$.
\end{proof}

We now turn our attention to the existence of stationary solutions with exactly one zero inside the interval $(0,\ell)$. The following proposition shows that solutions of this type can be constructed starting from the positive and negative ones that we investigate above in Proposition \ref{monotonic} and Corollary \ref{risolutivo}; as anticipated, their monotonic properties  will be determinant in order to obtain classical $C^2$ solutions. 
\begin{prop}\label{onezerosol}
For any $0<\e< \gamma \, (\ell/ 2\pi)^{2}$, with $\gamma$ being as in Proposition \ref{propesistenza}, problem \eqref{stationaryprob} admits a unique solution, named here $u_{1^-,\e}$, such that $u_{1^-,\e}(\ell/2)=0$ and
\begin{equation*}
u_{1^-,\e}(x) <0 \quad \mbox{for} \quad x < {\ell}/{2} \qquad \mbox{and} \qquad u_{1^-,\e}(x) >0 \quad \mbox{for} \quad x > {\ell}/{2}.
\end{equation*}
Moreover $ u'_{1^-,\e}(x)  \leq 1$ for all $x \in (0,\ell)$ and
\begin{equation*}
 u''_{1^-,\e}(x) >0 \quad \mbox{for} \quad x < {\ell}/{2} \qquad \mbox{and} \qquad  u''_{1^-,\e}(x) <0 \quad \mbox{for} \quad x > {\ell}/{2}.
\end{equation*}
Similarly, problem \eqref{stationaryprob} admits a unique solution $u_{1^+,\e}$ such that $u_{1^+,\e}(\ell/2)=0$ and
\begin{equation*}
u_{1^+,\e}(x) >0 \quad \mbox{for} \quad x < {\ell}/{2} \qquad \mbox{and} \qquad u_{1^+,\e}(x) <0 \quad \mbox{for} \quad x > {\ell}/{2}.
\end{equation*}
Moreover $ u'_{1^+,\e}(x)  \leq 1$ for all $x \in (0,\ell)$ and
\begin{equation*}
 u''_{1^+,\e}(x) <0 \quad \mbox{for} \quad x < {\ell}/{2} \qquad \mbox{and} \qquad  u''_{1^+,\e}(x) >0 \quad \mbox{for} \quad x > {\ell}/{2}.
\end{equation*}
Additionally
$$x-\ell/2 < u_{1^-,\e}(x) < 0 \quad \mbox{for} \quad x <\ell/2 \qquad \mbox{and} \qquad 0 < u_{1^-,\e}(x) < x-\ell/2 \quad \mbox{for} \quad x >l/2,$$
while
$$0 < u_{1^+,\e}(x) < x \quad \mbox{for} \quad x <\ell/2 \qquad \mbox{and} \qquad x-\ell < u_{1^+,\e}(x) < 0 \quad \mbox{for} \quad x >\ell/2.$$
\end{prop}

\begin{proof}
We start by proving the existence of $u_{1^-,\e}$. For this,
let $x_0 \in (0,\ell)$ and let us consider problem \eqref{probsuaeb} with the choice $a=x_0$ and $b=\ell$. 
As remarked above, there exists a positive solution to such a problem that we denote by $u_{+,\varepsilon}(x;x_0,\ell)$: to convince the reader, we repeat the same arguments of the proof of Proposition \ref{positivesolut} with $v_1(x)=x-x_0$ as super-solution and 
\begin{equation}\label{v2pos}
v_2(x)=\alpha \sin \left[\frac{\pi}{\ell-x_0}(x-x_0) \right]
\end{equation}
as sub-solution, provided that $\alpha\pi/(\ell-x_0)\leq 1$. Analogously, proceeding as in the proof of Proposition \ref{esistelau-}, if we choose $v_1(x)=x-x_0$ as sub-solution and
\begin{equation}\label{v2neg}
v_2(x)=-\alpha\sin\left[\frac{\pi}{\ell-x_0}(x-x_0) \right]
\end{equation}
as super-solution, we can prove the existence of a negative solution $u_{-,\varepsilon}(x;0,x_0)$ to \eqref{probsuaeb} in the interval $(0,x_0)$. Hence the function defined in the whole interval $(0,\ell)$ as follows
\begin{equation}\label{uM}
    u_{1^-,\e}(x) := \left\{ \begin{aligned} &u_{+,\varepsilon}(x;x_0,\ell) \quad \mbox{for} \quad x\geq x_0, \\ &u_{-,\e}(x;0,x_0)\quad \mbox{for} \quad x<x_0, \end{aligned}\right.
\end{equation}
solves the stationary problem \eqref{stationaryprob}.
In particular it must be $x_0=\ell/2$ in order for \eqref{uM} to be a classical $C^2$ solution; indeed, the first order derivative is continuous at $x=\ell/2$ in view of Corollary \ref{risolutivo} if and only if $x_0-0=\ell-x_0$, while the second order derivative is zero since $h'(0)=0$. Moreover, since Proposition \ref{propesistenza} has to hold in $(0, \ell/2)$, we need $\e < \alpha (\ell/2\pi)^2$ as required.
Going further, $u'_{1^-,\e}(x)  \leq 1$ for all $x \in (0,\ell)$ because so it is for $u_{+,\e}(x;\ell/2,\ell)$ and $u_{-,\e}(x;0,\ell/2)$ (see again the proofs of Propositions \ref{positivesolut} and \ref{esistelau-}); the convexity properties of $u_{1^-,\e}$ directly follow from the ones of $u_{\pm,\varepsilon}$ as well. Finally, since
\begin{equation} \label{x-ell2}
    x-\ell/2 < u_{-,\e}(x;0,\ell/2) < 0 \qquad \mbox{and} \qquad 0 < u_{+,\e}(x;\ell/2,\ell) < x-\ell/2,
\end{equation}
all the properties regarding the solution $u_{1^-,\e}$ are proved.

Existence and properties of the steady state $u_{1^+,\e}$ can be proven in a completely similar way by means of existence of a positive solution in the interval $(0,x_0)$ and a negative one in the interval $(x_0,\ell)$. In particular, one can reason exactly as in the proofs of Propositions \ref{positivesolut} and \ref{esistelau-} by choosing $v_1(x)=x$ as super-solution and \eqref{v2pos} as sub-solution in the interval $(0,x_0)$, $v_1(x)=x-\ell$ as sub-solution and \eqref{v2neg} as super-solution in the interval $(x_0,\ell)$. Again, one has that $x_0=\ell/2$ in order for $u_{1^+,\e}$ to solve the stationary problem in a classical sense. 
\end{proof}

\begin{remark}\label{morezeros}
As already observed in Remark \ref{rem1}, if in addition $h$ is symmetric one can find out that
$$u_{1^+,\e}(x)= -u_{1^-,\e}(\ell-x), $$
and thus deducing the properties of $u_{1^+,\e}$ directly from the ones of $u_{1^-,\e}$. 

The steady states described in Propositions \ref{positivesolut}, \ref{esistelau-}  and \ref{onezerosol} are depicted in Figure \ref{figSTAZ}.
\begin{figure}[tbp]
\begin{center}
\includegraphics[scale=0.4]{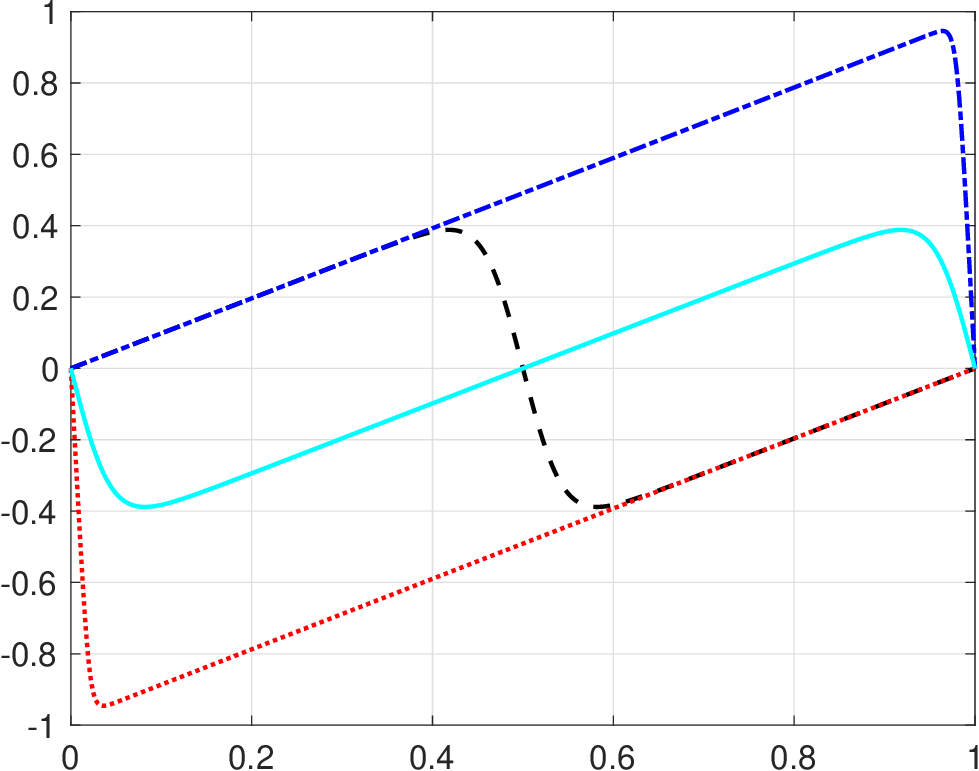}
\caption{\small{Plot of the  four stationary states, solutions to \eqref{stationaryprob} with $\varepsilon=0.01$, $h(u)= e^{-u^2}$ and $f(u)=u^2/2$.}}
\label{figSTAZ}
\end{center}
\end{figure}

For the sake of completeness, we also recall that, by proceeding as in the proof of Proposition \ref{onezerosol}, one can prove the existence of  solutions to \eqref{stationaryprob} with an arbitrary (but finite) number $N$ of zeroes inside the interval $(0,\ell)$. The idea   is  to ``glue" together a finite number of positive and negative steady states, defined in sub-intervals  $(\alpha_i, \beta_i) \subset (0,\ell)$ with $i=0,\dots,N$ such that $\alpha_0=0$, $\beta_{N}=\ell$, $\alpha_{i+1}=\beta_i$ for all $i=0, \dots , N-1$,
$$\bigcup_{i=0}^{N} \, (\alpha_i, \beta_i) = (0,\ell) \quad \mbox{and} \quad  \bigcap _{i=0}^{N} \,  (\alpha_i, \beta_i) = \emptyset. $$
One will obtain stationary solutions changing sign and convexity exactly {$N$ times} inside the interval $(0,\ell)$, having exactly $N$ internal zeros located at $\alpha_{i+1}=\beta_i$, for all $i=0, \dots \, N-1$.  We remark that in this case we need to ask $\e < \gamma(\ell/2^N \pi)^2$. 
\end{remark}

\section{Behaviour of the steady states}\label{sec3}
\subsection{Properties of the steady states with respect to $\e$}\label{epscomportamento}
In this subsection, we prove a convergence-type result as $\varepsilon\to 0^+$ for the positive and negative steady states $u_{\pm,\varepsilon}$ and, consequently, for the  steady states $u_{1^-,\e}$ and $u_{1^+,\e}$.
To this aim, we premise the following lemma in which a decreasing monotonic behaviour with respect to the parameter $\varepsilon$ is shown to hold. 

\begin{lemma}\label{lemmadecrescenza}
Let $u_{+,\varepsilon}$ be the unique positive solution to problem \eqref{stationaryprob} and let $u_{+,\varepsilon_1}$ be the unique positive solution to problem \eqref{stationaryprob} with $\varepsilon_1$ in place of $\varepsilon$. If $\varepsilon<\varepsilon_1$, then 
\begin{equation}\label{P1}
u_{+,\varepsilon_1}(x)\leq u_{+,\varepsilon}(x)\quad \text{for all $x \in (0,\ell)$}.
\end{equation}
\end{lemma}
\begin{proof}
From now on we denote with $L_\e$ the operator $L$ introduced in \eqref{operatoreL} and with $L_{\e_1}$ the same operator with $\e_1$ in place of $\varepsilon$. 
Then it holds that 
$$L_\varepsilon u_{+,{\varepsilon_1}}  \leq L_{\varepsilon_1} u_{+,{\varepsilon_1}}=0,$$
since $\varepsilon<\varepsilon_1$, $h(u_{+,{\varepsilon_1}})>0$, $h'(u_{+,{\varepsilon_1}})\leq0$ by \eqref{segnodih}, $u_{+,{\varepsilon_1}}''\leq 0$ and $u_{+,{\varepsilon_1}}$ solves \eqref{stationaryprob} with $\varepsilon_1 $ in place of $\varepsilon$. Thus $u_{+,{\varepsilon_1}}$ turns out to be a sub-solution to \eqref{stationaryprob}. Hence, arguing as in the proof of Proposition \ref{positivesolut}, if we take $v_1(x)=x$ and $v_2(x)= u_{+,{\varepsilon_1}}(x)$ as super-solution and sub-solution to problem \eqref{stationaryprob} respectively, then there exists a solution $\tilde{u}_{+,\varepsilon}$ to \eqref{stationaryprob} such that $u_{+,{\varepsilon_1}}\leq \tilde{u}_{+,\varepsilon}\leq x$, whence $\tilde{u}_{+,\varepsilon}$ is positive since $u_{+,{\varepsilon_1}}$ is positive. By uniqueness of the positive solution we infer that $\tilde{u}_{+,\varepsilon}=u_{+,\varepsilon}$ and the thesis trivially follows. 
\end{proof}
Similarly to the previous lemma, we can prove that the negative solution $u_{-,\e}$ of Proposition \ref{esistelau-} exhibits an increasing monotonic behaviour with respect to $\e$.

\begin{lemma}\label{lemmacrescenza}
Let $u_{-,\varepsilon}$ be the unique negative solution to problem \eqref{stationaryprob} and let $u_{-,\varepsilon_1}$ be the unique negative solution to problem \eqref{stationaryprob} with $\varepsilon_1$ in place of $\varepsilon$. If $\varepsilon<\varepsilon_1$, then 
\begin{equation}\label{P2}
u_{-,\varepsilon}(x)\leq u_{-,\varepsilon_1}(x)\quad \text{for all $x \in (0,\ell)$}. 
\end{equation}
\end{lemma}

\begin{proof}
The proof is analogous to the one of Lemma \ref{lemmadecrescenza}, based on the fact that this time
$$L_\varepsilon u_{-,{\varepsilon_1}}  \geq L_{\varepsilon_1} u_{-,{\varepsilon_1}}=0.$$
Indeed, in this case, $h'(u_{-,\e})\geq0$ by \eqref{segnodih}, while $u_{-,{\varepsilon_1}}''\geq 0$. Hence, by choosing $u_{-,{\varepsilon_1}}$ as super-solution and $x-\ell$ as sub-solution, we end up with $x-\ell \leq u_{-,{\varepsilon}} \leq u_{-,{\varepsilon_1}}$ and thus the proof is complete.
\end{proof}

The results of Lemmas \ref{lemmadecrescenza}-\ref{lemmacrescenza} (and, in particular, properties \eqref{P1} and \eqref{P2}) are depicted in Figure \ref{figLemma}, where we plot $u_{\pm,\e}$ for different values of $\e$; in particular one can see how $u_{+,\e}$ decreases ($u_{-,\e}$ increases) if $\e$ increases.
\begin{figure}[htbp]
\includegraphics[scale=0.4]{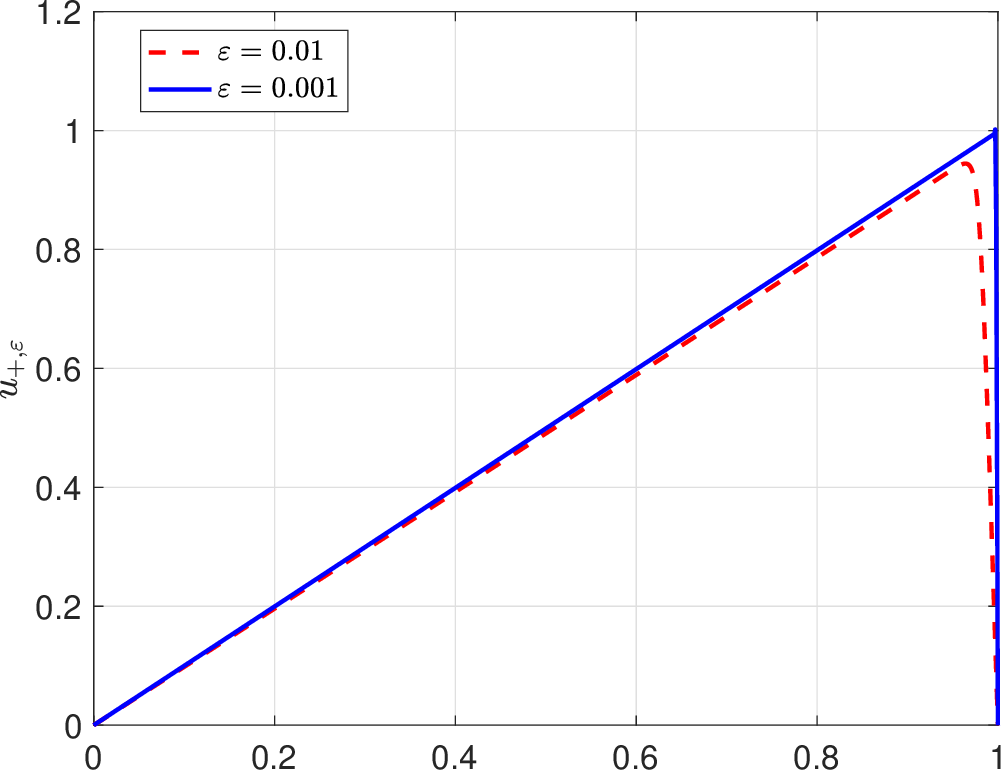}
\includegraphics[scale=0.4]{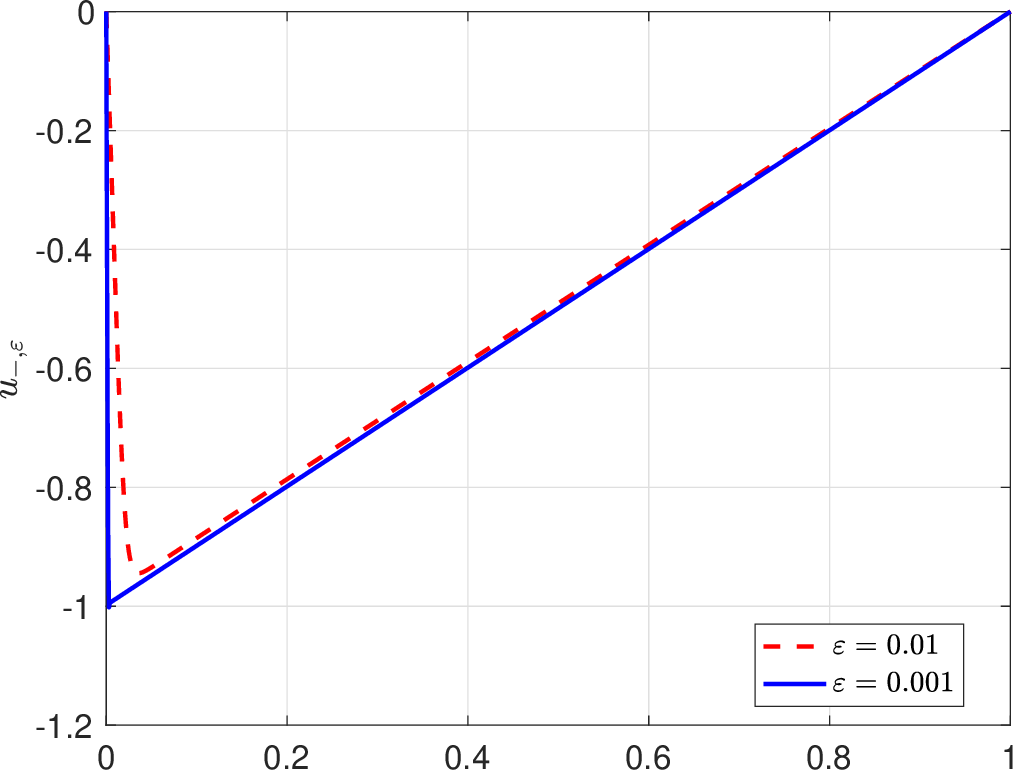}
\caption{\small{Plots of the solutions $u_{+,\e}$ and $u_{-,\e}$ (from left to right respectively) to \eqref{stationaryprob} with $h(u)=e^{-u^2}$ and $f(u)=u^2/2$.  }}
\label{figLemma}
\end{figure}

\begin{remark}\label{crescenzau-}
In the case where  $h$ is also symmetric, and hence $u_{-,\e}$ is explicitly given by \eqref{u-} (see Remark \ref{rem1}), one can derive the result of Lemma \ref{lemmacrescenza} by using the monotonic behaviour of  $u_{+,\e}(\ell-x)$ with respect to $\e$.
Indeed, one can easily find out that $u_{+,\e}(\ell-x)$ is the unique positive solution to 
\begin{equation}\label{nuovoperatore}
    -\e (h(u)u_x)_x -f'(u)(u_x+1)=0. 
\end{equation}
Therefore, setting $ \tilde{L}_\e u:=-\e (h(u)u_x)_x -f'(u)(u_x+1)$, it holds that 
$$\tilde{L}_\e (u_{+,\e_1}\circ (\ell- \cdot))\leq \tilde{L}_{\e_1} (u_{+,\e_1}\circ (\ell- \cdot))=0.$$
Hence we can conclude that $u_{+,\e_1}(\ell-x)$ is a sub-solution to problem \eqref{nuovoperatore}. Taking $\ell-x$ as super-solution (in fact $h'(\ell-x)\leq0$ by \eqref{segnodih}), we obtain that there exists a positive solution between $u_{+,\e_1}(\ell-x)$ and $\ell-x$, which in turn by uniqueness coincides with $u_{+,\e}(\ell-x)$. Thus for all $x\in (0,\ell)$ 
\begin{equation*}
u_{+,\e_1}(\ell-x) \leq u_{+,\e}(\ell-x),
\end{equation*}
and consequently we can infer that $u_{-,\e}(x)\leq u_{-,\e_1}(x)$ for all $x\in (0,\ell)$ .
\end{remark}
In the next propositions, we give a complete description of the (uniform) convergence of the steady states with respect to the parameter $\e$.

\begin{prop}\label{convergenzau+}
Let $u_{\pm,\e}$ be as in Proposition \ref{positivesolut} and in Proposition \ref{esistelau-}, respectively. Then, it holds that 
\begin{equation*}\label{limiteu+}
    \lim_{\varepsilon\to 0^+} u_{+,\varepsilon}(x)=x \qquad \mbox{and} \qquad  \lim_{\varepsilon\to 0^+} u_{-,\varepsilon}(x)=x-\ell,
\end{equation*}
uniformly on compact sets of $[0,\ell)$.
\end{prop}

\begin{proof}
Thanks to Lemma \ref{lemmadecrescenza}, we obtain that the pointwise limit of $u_{+,\varepsilon}(x)$ in $(0,\ell)$ as $\varepsilon\to 0^+$ does exist and, denoting by 
$\Bar{u}(x)$ such a limit, we have that $\Bar{u}(x)\leq x$ in virtue of Proposition \ref{positivesolut}. In order to prove that $\Bar{u}(x)$ coincides with $x$, we build a sub-solution to problem  \eqref{stationaryprob} for $\e$ small enough, which is arbitrarily close to $x$. 
To this aim, we consider two positive numbers $a<\ell$ and $\lambda<1$ (to be taken later as close as possible to $\ell$ and $1$ respectively) and we define the following function 
\begin{equation*}
    v_{a,\lambda}(x) := 
    \begin{cases}
    	\lambda x, \qquad\qquad  & x\in [0,a], \\ 
	\xi(x), &  x\in [a,b], \\  
	\displaystyle\frac{\lambda a}{\ell-b} (\ell-x), & x\in [b,\ell],
    \end{cases}
\end{equation*}
where $b=(\ell+a)/2$, $\xi$ is a function of class $C^2$ such that 
\begin{align*}
	&\xi(a)=\xi(b)= \lambda a, \qquad  \xi'(a)=\lambda, \qquad  \xi'(b)=-\frac{\lambda a}{\ell-b},\\
 	&\xi''(a)=\xi''(b)=0 \qquad \text{ and }\quad  \xi''(x)\leq 0, \quad  \text{ for all }\ x\in [a,b].
\end{align*} 
In particular, we deduce that $\xi'(x) \leq 1$ for every $x\in [a,b]$. 
Putting together all these information, we can easily check that for $\e$ small enough in dependence on $a$ and $\lambda$ 
\begin{equation*}\label{subsolutionproof}
L v_{a,\lambda} = -\varepsilon h'(v_{a,\lambda})(v_{a,\lambda})_x^2-\varepsilon h(v_{a,\lambda})(v_{a,\lambda})_{xx} +f'(v_{a,\lambda})[(v_{a,\lambda})_x-1]\leq 0,
\end{equation*}
being $L$ as in \eqref{operatoreL}. Indeed, the first term in the above expression is positive since $h'(v_{a,\lambda})\leq0$ by \eqref{segnodih}, while the third term is negative since $(v_{a,\lambda})_x <1$ by construction; the  second term  vanishes if $x \in [0,a] \cup [b,\ell]$ and is positive for $x \in [a,b]$. Thus we can conclude that $v_{a,\lambda}$ is a sub-solution to problem \eqref{stationaryprob} provided that $\e$ is small enough in dependence on $a$ and $\lambda$. 
If in addition we take into account that $v_{a,\lambda}(x)\leq x $ for every $x\in [0,\ell]$ and that $x$ is a super-solution to problem \eqref{stationaryprob}, by uniqueness of the positive solution we can conclude that
\begin{equation*}
\sup_{x\in [0,\ell)} |u_{+,\e}(x)-x|\leq \sup_{x\in [0,\ell)}  |v_{a,\lambda}(x)-x|,
\end{equation*}
for $\e>0$ conveniently small in dependence on $a$ and $\lambda$.
Since the right hand side is arbitrarily small choosing $a$ and $\lambda$ sufficiently close to $\ell$ and 1 respectively, the above inequality implies the uniform convergence of $u_{+,\e}$ to $x$ as $\e\to 0^+$ and thus that $\Bar{u}(x)= x$. 
The proof for $u_{-,\e}$ is completely analogous and we omit it for shortness.
\end{proof}

\noindent Similarly, as a corollary of the previous proposition, we can derive a convergence-type result for the steady states $u_{1^-,\e}$ and  $u_{1^+,\e}$ as well.
\begin{corollary}\label{convuM}
For all $x\in (0,\ell)$ it holds that, as $\varepsilon \to 0^+$ 
\begin{itemize}
 \item the solution $u_{1^-,\e}$ converges pointwise to the function $x-\frac{\ell}{2}$;
\item the solution $u_{1^+,\e}$ converges pointwise to the function $x \, \chi_{(0,\ell/2)} + (x-\ell) \, \chi_{(\ell/2,\ell)}$.
\end{itemize}
Moreover, the convergences are uniform on compact sets of $(0,\ell)$ and $\left[0,\frac{\ell}{2}\right) \cup \left(\frac{\ell}{2}, \ell\right)$ respectively.
\end{corollary}

\begin{proof}
For the sake of brevity, we focus on the function $u_{1^-,\e}$, giving only a hint of the proof of its convergence since it can be easily deduced by appropriately modifying the proof of Proposition \ref{convergenzau+}. 
More precisely, we study separately the convergence in the two intervals $(0,\ell/2)$ and $(\ell/2,\ell)$: 
by the way $u_{1^-,\e}$ is defined, thanks to the monotonicity properties of $u_{+,\e}$ and $u_{-,\e}$ provided in Lemmas \ref{lemmadecrescenza} and \ref{lemmacrescenza} respectively, and in virtue of \eqref{x-ell2}, we have that the limit of $u_{1^-,\e}$ as $\e\to 0^+$ exists and it is greater than $x-\ell/2$ for $x\in (0,\ell/2)$ and less than $x-\ell/2$ for $x\in (\ell/2,\ell)$. 
Therefore, arguing as in the proof of Proposition \ref{convergenzau+}, in order to prove that the limit is exactly equal to $x-\ell/2$, we construct a super-solution in the interval $(0,\ell/2)$ and a sub-solution in $(\ell/2,\ell)$ for $\e$ sufficiently small. 
Specifically, in the latter interval we get inspiration from the function $v_{a,\lambda}$ of the proof of Proposition \ref{convergenzau+}: we take 
$a\in (\ell/2,\ell)$, $b=(\ell+a)/2$ and $\lambda\in (0,1)$, and define  
\begin{equation*}
    \bar{v}_{a,\lambda}(x) :=
    \begin{cases}
    	\lambda (x-\ell/2), \qquad \quad & x\in [\ell/2,a], \\ 
	\xi(x), & x\in [a,b], \\  
	\displaystyle\frac{\lambda (a-\ell/2)}{\ell-b} (\ell-x), & x\in [b,\ell],
    \end{cases}
\end{equation*}
where $\xi$ is of class $C^2$ and satisfies 
\begin{align*}
	&\xi(a)=\xi(b)= \lambda(a-\ell/2), \qquad  \xi'(a)=\lambda,\qquad  \xi'(b)=\frac{\lambda(\ell/2-a)}{\ell-b},\\
	& \xi''(a)=\xi''(b)=0 \qquad \text{ and }\quad  \xi''(x)\leq 0, \quad \text{for all}\ x\in [a,b].
\end{align*}
By symmetry, we can construct a super-solution in the interval $(0,\ell/2)$ as follows
\begin{equation*}
    	\tilde{v}_{a,\lambda}(x) := 
    	\begin{cases}
    		\displaystyle\frac{\lambda (\ell/2-a)}{\ell-b}\, x, \qquad & x\in [0,\ell-b], \\ 
		\zeta(x), & x\in [\ell-b,\ell-a], \\  
		\lambda (x-\ell/2), & x\in [\ell-a,\ell/2],
    	\end{cases}
\end{equation*}
where $\zeta(x)=-\xi(\ell-x)$ and hence
\begin{align*}
	&\zeta(\ell-b)=\zeta(\ell-a)=\lambda(\ell/2-a), \qquad   \zeta'(\ell-b)= \frac{\lambda(\ell/2-a)}{\ell-b},\qquad \zeta'(\ell-a)=\lambda,\\
	&\zeta''(\ell-b)=\zeta''(\ell-a)=0, \qquad \zeta''(x)\geq 0, \quad \text{ for any }\ x\in [\ell-b,\ell-a].
\end{align*}
The proof is thereby complete.
\end{proof}
The results of Corollary \ref{convuM} are illustrated in Figure \ref{figconvergenza} (with $\ell=1$); precisely one can see that, the more $\e \to 0$, the more $ u_{1^-,\e}$ is close to $x-1$ (right picture), while $u_{1^+,\e}$  to the function $x \, \chi_{(0,1/2)} + (x-1) \, \chi_{(1/2,1)}$ (left picture).

\begin{figure}[htbp]
\begin{center}
\includegraphics[scale=0.4]{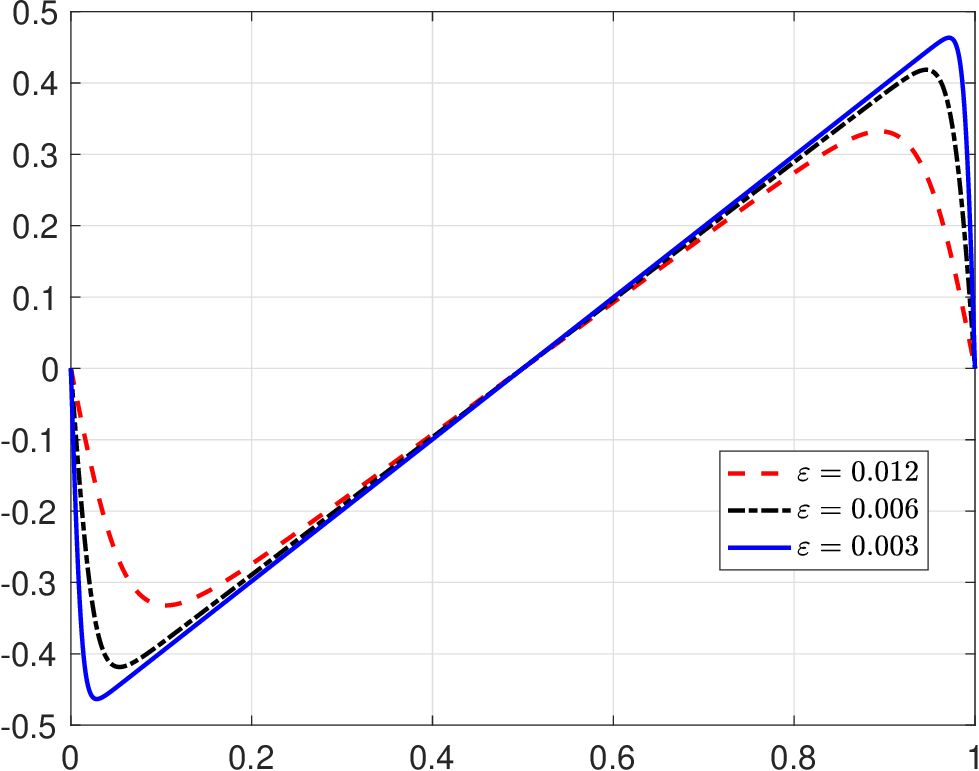}
\includegraphics[scale=0.4]{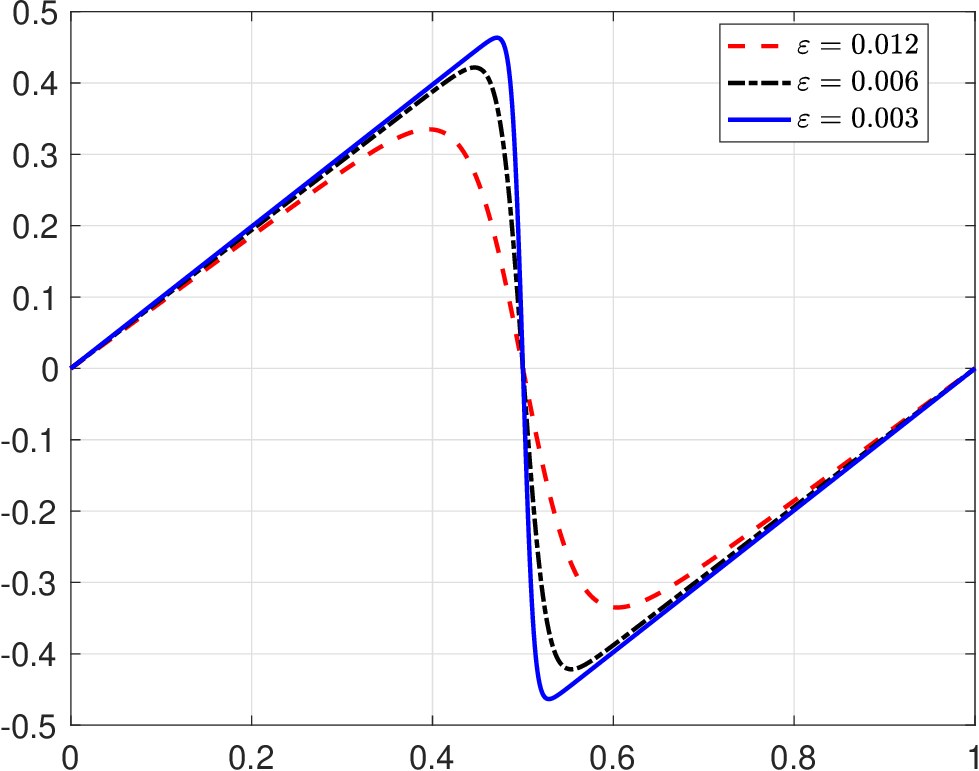}
\end{center}
\caption{\small{Plots of $u_{1^-,\e}$ (right picture) and $u_{1^+,\e}$ (left picture) for different values of $\e$, with $h(u)=e^{-u^2}$ and $f(u)=u^2/2$.}}
\label{figconvergenza}
\end{figure}

In the following proposition, we improve the result contained in Proposition \ref{convergenzau+} above, giving a precise estimate of the behaviour of $u'_{+,\e}$ at $x=\ell$ and of $u'_{-,\e}$ at $x=0$ as $\e$ becomes smaller. Indeed, since  $u_{+,\e}$ is positive and $u_{+,\varepsilon}(\ell)=0$, we already know that $u_{+,\varepsilon}'(\ell)<0$,  as well as $u_{-,\varepsilon}'(0)<0$; we additionally prove that
both $u'_{+,\e}(\ell)$ and $u'_{-,\e}(0)$ behave like $-\e^{-1}$ as $\e\to 0^+$. 
\begin{prop}\label{velocitaconv}

There exist two positive constants $\beta_1, \beta_2>0$ such that, for $\e$ small enough
\begin{equation}\label{stima11}
-\frac{\beta_1}{\varepsilon}  \leq   u'_{+,\varepsilon}(\ell) \leq -\frac{\beta_2}{\varepsilon}.
\end{equation}
Analogously, there exist two positive constants $\gamma_1, \gamma_2>0$ such that, for $\e$ small enough
\begin{equation}\label{stima12}
-\frac{\gamma_1}{\varepsilon}  \leq   u'_{-,\varepsilon}(0) \leq -\frac{\gamma_2}{\varepsilon}.
\end{equation}
\end{prop}
\begin{proof}
As already shown in the proof of Proposition \ref{positivesolut}, $u_{+,\e}$ has a unique maximum point located at some point $x_\e\in (0,\ell)$; denoted with $M_\e = u_{+,\e}(x_\e)$, 
we observe that $u_{+,\e}$ is invertible in the interval $[x_\varepsilon, \ell]$, so that we are allowed to consider the function $u^{-1}_{+,\e}$ from $[0,M_\varepsilon]$ to $[x_\varepsilon, \ell]$. 
We introduce a new function $q$ of the variable $y\in [0,M_\varepsilon]$ such that 
$$q( y)= u'_{+,\e}( x) \quad \mbox{where} \quad  x = u_{+,\e}^{-1}( y)\in [x_\e,\ell].$$ 
In this way $q(0)=u'_{+,\e}(\ell)$ since $u_{+,\varepsilon}^{-1}(0)\in \{0,\ell\}$ and we are inverting the function $u_{+,\e}$ in the interval $[x_\varepsilon,\ell]$, and $q(M_\e)=u'_{+,\e}(x_\e)=0$.
Exploiting the above change of variable, \eqref{stationaryprob} turns into
\begin{equation}\label{SProof'}
   -\e(h(y)q(y))_x+f'(y)(q(y)-1)=0. 
\end{equation}
If we now set $z(y)=h(y)q(y)$, then we obtain 
$$\left(z(y) \right)_x= z'(u_{+,\e}(x))(u_{+,\e})_x=z'(y) q(y) = z'(y) \, \frac{z(y)}{h(y)}.$$
Plugging this into \eqref{SProof'}, we obtain that
\begin{equation}\label{lanumero}
    -\e z'(y) \frac{z(y)}{h(y)}+ \frac{f'(y)}{h(y)} \left(z(y)-h(y) \right)=0,
\end{equation}
from which in turn we deduce the following inequality 
\begin{equation}\label{ciao'}
   \e z(y) z'(y)=f'(y)(z(y)-h(y)) \geq f'(y)(z(y)-M_1), 
\end{equation}
being $$M_1 := \max_{x \in [0,\ell]} h(u_{+,\e}(x)).$$ At this point we remark that $z(y)\leq M_1$ for all $y \in [0,M_\varepsilon]$: indeed, for all $y \in [0,M_\varepsilon]$, $q(y)=u'_{+,\e}(x)\leq 1$  as a consequence of Proposition \ref{positivesolut} and $h(y)\leq M_1$ by definition of $M_1$. For this, we divide each member of \eqref{ciao'} by $M_1-z(y)$ in order to preserve the inequality. 
Integrating \eqref{ciao'} between $0$ and $M_\e$, by separation of variables,  we get
\begin{equation*}
    \int_0^{M_\e} \frac{z' z}{M_1-z} \, dy \geq -\frac{1}{\e} \int_0^{M_\e} f'(y) \, dy,
\end{equation*}
which in turn thanks to \eqref{segnodif} leads to 
\begin{equation*}\label{proof2'}
      \e[z(y)+ M_1 \ln(M_1-z(y))  ]^{y=M_\e}_{y=0} \leq  f(M_\e).
\end{equation*}
From this, recalling that $z(M_\e)= 0$ and $z(0)=M_1 u'_{+,\e}(\ell)$ since $M_1=h(0)$, we obtain 
$$-\e M_1 u'_{+, \e }(\ell)+ \e M_1\ln(M_1)+ \e M_1 \ln \left( \frac{1}{M_1 - M_1 u'_{+, \e }(\ell)}\right) \leq {f(M_\e)},$$
which, rewriting the left-hand side in a proper way, can be rephrased as follows
\begin{equation}\label{disug'}
\begin{aligned}
-\e M_1 u'_{+, \e }(\ell) \left[1+ \frac{\ln \left(1- u'_{+, \e }(\ell)\right)}{u'_{+, \e }(\ell)}  \right] \leq  {f(M_\e)}.
\end{aligned}
\end{equation}
At this point we note that the function
$$g(s)= \frac{\ln \left(1- s\right)}{ s}, \qquad s\in (-\infty,0)$$
is negative and such that 
$$\lim_{s \to 0^-} g(s) = -1 \qquad \mbox{and} \qquad \lim_{s \to -\infty} g(s)=0.$$
Hence $$\sup_{s\in (-\infty,0)}\left|g(s)\right|=1;$$ from this and \eqref{disug'}, we infer that there exists a positive constant $c_1>0$ such that 
$$-\e M_1 u'_{+,\e}(\ell)\,c_1 \leq \, f(M_\e),$$
that is, 
$$u'_{+,\e}(\ell) \geq - \frac{  f(M_\e)}{\e \,c_1\,M_1}.$$
From this we can conclude that there exists $\beta_1>0$ such that for $\e$ sufficiently small $$u'_{+,\e}(\ell) \geq - \frac{  \beta_1}{\e }.$$  

In order to obtain the inverse inequality, we carry out the same computations by using 
$$m_1:=\min_{x \in [0,\ell]} h(u _{+,\e}(x)); $$
namely, referring to \eqref{lanumero}, we have that
\begin{equation*}
    \varepsilon z(y) z'(y)=f'(y)(z(y)-h(y))\leq f'(y)(z(y)-m_1).
\end{equation*}
Observing that $m_1-z$ is positive since $u'_{+,\e}(x)\leq 0$ for every $x\in [x_\e,\ell]$, we divide each member by $m_1-z(y)$ and, integrating between 0 and $M_\e$, we deduce that
\begin{equation*}
    [z(y)+m_1 \ln (m_1-z(y))]^{y=M_\e}_{y=0} \geq  \frac{f(M_\e)}{\e}.
\end{equation*}
Hence it holds that
\begin{equation*}
   -M_1 u'_{+,\e}(\ell) + m_1\ln m_1-m_1\ln (m_1 - M_1 u'_{+,\e}(\ell))\geq \frac{f(M_\e)}{\e},
\end{equation*}
which can be rewritten as follows 
\begin{equation*}\label{eq56}
    -M_1 u'_{+,\e}(\ell) \left[1+ \frac{\ln \left(1-\frac{M_1 u'_{+,\e}(\ell)}{m_1}\right)}{\frac{ M_1 u'_{+,\e}(\ell)}{m_1}} \right] \geq \frac{f(M_\e)}{\e}.  
\end{equation*}
As before, we can infer that there exists a positive constant $c_2>0$ such that 
$$ -M_1 u'_{+,\e}(\ell) c_2\geq  \frac{f(M_\e)}{\e},$$
and thus there exists $\beta_2>0$ such that for $\e$ sufficiently small it holds that
$$u'_{+,\e}(\ell) \leq -\frac{\beta_2}{\e}.$$
The proof of \eqref{stima11} is thereby complete. As for the proof of \eqref{stima12}, it is enough to take, in place of $M_\e$, the unique minimum of $u_{-,\e}$ denoted for instance with $m_\e$ and let now $x_\e\in (0,\ell)$ be such that $u_{-,\e}(x_\e)=m_\e$; after that, we introduce for every $y\in [0,m_\e] $ the function $p(y)$ defined as $u'_{-,\e}(x)$, where $x=u_{-,\e}^{-1}(y)\in [0, x_\e]$. 
Defining the function $z(y):= h(y)p(y)$ and reasoning precisely as above, we can conclude the proof of the proposition. 
\end{proof}

\subsection{Stability properties of the steady states}
In the present subsection we prove on one hand the stability of the steady states $u_{\pm,\e}$, and on the other hand the instability of $u_{1^-,\e}$ and $u_{1^+,\e}$.
In the following proposition we exhibit a super-solution and two sub-solutions that will be employed to prove the results below in Theorem \ref{stabthm}. 
\begin{prop}\label{preliminary1}
The function $v_1(x)= A x + B $ is a super-solution to \eqref{stationaryprob} for all $A >1$ and $B>0$. Moreover, for all $0 \leq a \leq  b \leq \ell$ the function
$$v_2(x) := \begin{cases} u_{-,\e}(x;0,a), \qquad & x \in [0,a), \\ 
u_{+,\e}(x;a,b), & x \in [a,b), \\
0,& x \in [b,\ell],\end{cases}$$
is a sub-solution  to \eqref{stationaryprob} if $a \leq b-a$; in addition, if $\ell-b<b-a$, also the function
$$v_3(x) := \begin{cases} u_{-,\e}(x;0,a), \qquad & x \in [0,a), \\ 
u_{+,\e}(x;a,b),& x \in [a,b), \\
u_{-,\e}(x;b,\ell), & x \in [b,\ell],\end{cases}$$
is a sub-solution  to \eqref{stationaryprob}.
\end{prop}

\begin{proof}
If we compute the operator $L$ defined as in \eqref{operatoreL} at $v_1$, we have that
$$Lv_1=- \e \, A^2 h'(v_1)+ f'(v_1) (A-1) \geq 0,$$
since $A>1$ and in light of \eqref{segnodih} and \eqref{segnodif1}.
On the other side, we first notice that 
because of its construction the function $v_2$ satisfies $Lv_2=0$ on each of intervals $(0,a)$, $(a,b)$ and $(b,\ell)$.
Then, recalling Definition \ref{def_subsuper}, for any test function $\varphi \geq 0$ in $(0,\ell)$ such that $\varphi(0)=\varphi(\ell)=0$, we have that 
\begin{equation}\label{contiale}
    \begin{aligned}
     \int_0^\ell\{\varepsilon h(v_2)v_2' \varphi_x+&[f'(v_2)v_2'-f(v_2)]\varphi \}\, dx  \\
     &=\e h(v_2(a))v_2'(a^-)\varphi(a)\\
     &\quad +\e h(v_2(b))v_2'(b^-)\varphi(b)-\e h(v_2(a))v_2'(a^+)\varphi(a)\\
     &\leq 0,
    \end{aligned}
\end{equation}
where we have used that $v_2'(b^+)=0$, $v_2'(b^-)$ is negative and the difference $v_2'(a^-)-v_2'(a^+)$ is negative in light of Corollary \ref{risolutivo}; hence we can conclude that $v_2$ is a sub-solution. 
At last, in order to prove that  $v_3$ is a sub-solution, we make similar computations as the ones in \eqref{contiale}: in this case also the difference $v_3'(b^-)-v_3'(b^+)$ appears, and this quantity is negative by using again Corollary \ref{risolutivo}, since $\ell-b<b-a$.
\end{proof}
We now turn to the problem of stability/instability of the stationary solutions. Our main result is the following.
\begin{theorem}\label{stabthm}
Let $u$ be the solution to \eqref{prob} with initial datum $u_0$ satisfying 
\begin{equation}\label{conddatoiniziale}
v_3(x) \leq u_0(x) \leq v_1(x) \quad \mbox{for any} \ x \in [0,\ell],
\end{equation}
where $v_1$ and $v_3$ are defined as in Proposition \ref{preliminary1}. Then, for all $x \in [0,\ell]$
\begin{equation}\label{asylimit}
\lim_{t \to + \infty} u(x,t) = u_{+,\e}(x).
\end{equation}
Moreover, one can choose initial data $u_0$ arbitrarily close to $u_{1^-,\e}$ such that the corresponding solution satisfies either \eqref{asylimit} or
\begin{equation*}\label{asylimit2}
\lim_{t \to + \infty} u(x,t) = u_{-,\e}(x).
\end{equation*}
\end{theorem}

\begin{proof}
We start by proving the stability of $u_{+,\e}$; let $v_1(x,t)$ be the solution to \eqref{prob} arising from $u_0(x)=v_1(x)$, with $v_1$ being defined as in Proposition \ref{preliminary1}.  Notice that, by \eqref{propr-sol-pos}, we have $v_1(x) >u_{+,\e}(x)$ for all $x \in [0,\ell]$, so that, by using the comparison principle, we deduce that $v_1(x,t) >u_{+,\e}(x)$ for all $x \in [0,\ell]$ and $t >0$. 
Recalling that  solutions to \eqref{prob} with a sub- or a super-solution as initial datum converge monotonically to a stationary solution $\bar u$ as $t \to +\infty$ (see \cite{BKS} for the linear diffusion case with $f(u)=u^2/2$), we infer that $v_1(x,t)$ decreases in time to a stationary solution to \eqref{prob}, and the limit function as $t \to +\infty$ has to be $u_{+,\e}(x)$, that is the unique positive solution to \eqref{stationaryprob}.
Similarly, if $v_3(x,t)$ is the solution to \eqref{prob} with initial datum  $u_0(x)=v_3(x)$, then $v_3(x,t) <u_{+,\e}(x)$ for all $x \in [0,\ell]$ and $t >0$ and, as above, we can conclude that  $v_3(x,t)$ increases in time to $u_{+,\e}(x)$ as $t \to +\infty$; indeed, if $a$ and $\ell-b$ are small enough, there are no other stationary solutions greater then $v_3$.
Collecting the previous results, if the initial datum satisfies \eqref{conddatoiniziale}, then the corresponding solution satisfies 
$$v_3(x,t) \leq u(x,t) \leq v_1(x,t),$$
 implying  \eqref{asylimit}, and the first part of the proof is complete.

In order to prove the second part of the theorem, we consider $v_2$ defined as in  Proposition \ref{preliminary1}; at first we choose  $0<a\leq\frac{\ell}{2}$ and $b=\ell$, and we notice that, because of Proposition \ref{monotonic}, we have that $u_{1^-,\e}(x) \leq v_2(x)$ for all $x \in [0,\ell]$. Hence, reasoning as before, the solution $v_2(x,t)$ to \eqref{prob} with initial datum $u_0(x)=v_2(x)$ increases in time to $u_{+,\e}(x)$ as $t \to +\infty$.  Analogously, if $a \geq  \frac{\ell}{2}$, then $u_{1^-,\e}(x) \geq v_2(x)$ for all $x \in [0,\ell]$ and the solution $v_2(x,t)$ decreases in time to $u_{-,\e}(x)$ as $t \to +\infty$.
This proves that  $u_{1^-,\e}$ is unstable, since $v_2$ can be chosen arbitrarily close to $u_{1^-,\e}$ if $a \to \frac{\ell}{2}^{\pm}$. 
\end{proof}

\begin{remark}\label{INSTAB}
The stability of $u_{-,\e}$ and the instability of $u_{1^+,\e}$ can be proven in a completely symmetric way. 
\end{remark}

\section{Finite time behavior of solutions}\label{hyperbolic}
In this section it is convenient to reframe our problem on the whole real line, in order to take advantage of some known results. Let us consider the strip $S:= \{(x,t) \, : \, x \in \mathbb{R}, t \geq 0\}$, and let us extend the function $u_0$ on the whole real line by disparity and periodicity: precisely, we require $u_0$ to be odd about $x=0$, that is we extend the function in $(-\ell,0)$ by setting $u_0(x)=-u_0(-x)$, and after that we consider the function for all $x \in \mathbb{R}$ by letting $u_0(x)=u_0(x+2\ell)$. Let now $\bar u_\e$ be the solution to
\begin{equation}\label{prob3}
\begin{cases} u_t=\varepsilon(h(u)u_x)_x -f(u)_x + f'(u), \qquad    & x \in S, \\
u(x,0) = u_0(x), & x \in \mathbb{R}. 
\end{cases}	
\end{equation}
Problem \eqref{prob3} admits a unique solution, implying that $\bar u_\e(0,t)=\bar u_\e(\ell,t)=0$; hence $\bar u_\e(x,t) \equiv u_\e(x,t)$ for $x \in [0,\ell]$, where $u_\e$ is the solution to the original problem \eqref{prob}. From now on we will thus use $u_\e$ instead of $\bar u_\e$.

It is well known (see \cite{Nessy,Ole}) that for the solution to \eqref{prob3} it holds that
\begin{equation*}
\lim_{\e \to 0} u_\e(x,t) = U(x,t),
\end{equation*}
where $U$ is the entropy solution of the hyperbolic equation
\begin{equation}\label{prob3hyp}
U_t+f(U)_x-f'(U)=0, \qquad U(x,0)=u_0(x),
\end{equation}
and the convergence is in $L^1_{\mathrm{loc}}$; precisely, the following results was proved in \cite{Nessy}.
\begin{prop}\label{prop14}
Let $u_\e$ be the solution to \eqref{prob3} and suppose that $u_0 \in L^\infty(\mathbb{R})$ and that there exists a finite $M>0$ such that
\begin{equation*}
\frac{u_0(x)-y_0(y)}{x-y} \leq M, \qquad \mbox{for all} \quad x \neq y.
\end{equation*}
Then, as $\e \to 0$ and away from shock waves, $u_\e \to U$ in $C^0_{loc}$, where $U$ is the entropy solution to \eqref{prob3hyp} in $S$.
\end{prop}
Finally, $U$ (as well of $u_\e$) satisfies the boundary conditions $U(0,t)=U(\ell,t)=0$, and therefore solves the boundary value problem. 

\subsection{Asymptotic behavior for the hyperbolic problem} 
As stated in the introduction, we differentiate our results depending on the form of the initial datum; we consider three different types of data.

\begin{itemize}
\item Type A: $u_0 >0$ (or, equivalently $u_0<0$) in $(0,\ell)$.
\item Type B:  $u_0 <0$ in $(0,x_0)$ and $u_0>0$ in $(x_0,\ell)$ for some $x_0 \in (0,\ell)$.
\item Type C: $u_0 >0$ in $(0,x_0)$ and $u_0<0$ in $(x_0,\ell)$ for some $x_0 \in (0,\ell)$.
\end{itemize}

The asymptotic time behavior of the solutions to \eqref{prob3hyp} has been studied in \cite{lyberopoulos}; we recall the version of the result we will use here, that concerns initial data which change sign at most once.
\begin{prop}
Suppose that $u_0$ change sign at most once in $(0,\ell)$. Then the limiting behavior for $t \to +\infty$ of $U$, solution to \eqref{prob3hyp}, is given by one of the following alternatives:
\begin{equation*}
\begin{aligned}
(i) \, \, U(x,t) &\to x-x_0\   \mbox{as}\  t \to +\infty, \quad \mbox{for some} \  x_0 \in (0,\ell); \\
(ii) \,  \, U(x,t) &\to x \   \mbox{as}\  t \to +\infty;\\
(iii) \,  \, U(x,t) &\to x -1\   \mbox{as}\  t \to +\infty;\\
(iv) \,  \, U(x,t) &\to x \, \chi_{\left[0,\frac{\ell}{2} \right)} + (x-\ell) \chi_{\left(\frac{\ell}{2},\ell \right]}\   \mbox{as}\  t \to +\infty.
\end{aligned}
\end{equation*}
\end{prop}
Thanks to the previous proposition, we can prove the following result.
\begin{theorem}\label{th:asyhyp}
Let $U$ be the solution to \eqref{prob3hyp}. 
\begin{itemize}
\item If $u_0$ is of type A, then 
$$\lim_{t \to +\infty} U(x,t)=\begin{cases}  x, \quad    &\mbox{if} \quad  u_0 >0,  \\
 x-\ell,\quad &\mbox{if}  \quad u_0 <0,
 \end{cases} 
 $$
 uniformly on compact sets of $[0,\ell)$ and $(0,\ell]$ respectively.
 \item If $u_0$ is of type B, then 
 $$\lim_{t \to \infty} U(x,t)=x-x_0,$$
 uniformly on compact sets of $(0,\ell)$.
 \item If $u_0$ is of type C, then
 $$\lim_{t \to +\infty} U(x,t)=\begin{cases}  x, \qquad  \;  &  x_0 \in (0,\ell/2),  \\
 x\, \chi_{\left[0,\frac{\ell}{2} \right)} + (x-\ell) \chi_{\left(\frac{\ell}{2}, \, \ell \right]}, & x_0 =\frac{\ell}{2}, \\
 x-\ell,  &x_0 \in (\ell/2,\ell),
 \end{cases}
 $$
uniformly on compact sets of $[0,\ell)$,  $[0,\ell] \setminus  \left\{ \frac{\ell}{2}\right\} $ and $(0,\ell]$ respectively. 
 \end{itemize}
\end{theorem}
While case A directly follows by using the maximum principle, for the proof of case B the following lemma is required.
\begin{lemma}
Let $U:=U(t,x)$ be the solution to problem \eqref{prob3hyp} with initial datum $u_0$ of type B.
Then $U(t,x_0)=0$ for all $t\geq 0$ and $U(t,x)$ is continuous for all $t>0$ and for all $x\in (x_0-\gamma, x_0+\gamma)$ for some sufficiently small $\gamma>0$. 
\end{lemma}

\begin{proof}
This result is a generalization of \cite[Lemma 7.4]{BKS} and it is a consequence of the assumptions \eqref{segnodif} (in particular, it follows from \eqref{segnodif1}).
Let $U:=U(t,x)$ be the solution to problem \eqref{prob3hyp} with initial datum $u_0$ of type B;
denoting by $z(\tau,s):=U(t(\tau,s),x(\tau,s))$, we obtain the following system for the characteristic curves
\begin{equation*}
	\begin{cases}
		\frac{dt}{d\tau}=1,\\
		\frac{dx}{d\tau}=f'(z),\\
		\frac{dz}{d\tau}=f'(z),
	\end{cases}
	\qquad \qquad
	\begin{cases}
		t(0,s)=0,\\
		x(0,s)=s,\\
		z(0,s)=u_0(s).
	\end{cases}
\end{equation*}
Since $f'(z)=0$ if and only if $z=0$ and $u_0(x_0)=0$, the characteristic curve which starts at $(0,x_0)$ is vertical.
Moreover, the only reason for $U(t,x_0)$ to be different from zero may be the appearance of a shock wave at $x=x_0$.
However, notice that the velocity of the shock is given by 
$$\frac{ds}{dt}=\frac{f(u_2)-f(u_1)}{u_2-u_1},$$
where $u_1$ and $u_2$ are the limits of $U$ from both sides of the shock.
Therefore, the sign of $u_0$ and of $f'$ allows us to state that the shock waves may move only to the right if $x\in(x_0,\ell)$ and to the left if $x\in(0,x_0)$.
Hence, we can conclude that no shock wave can reach $x=x_0$ and, as a consequence, $U(t,x_0)=0$ for all $t\geq 0$ and $U(t,x)$ is continuous for all $t>0$
if $|x-x_0|<\gamma$ for some (sufficiently small) $\gamma>0$.
\end{proof}

\subsection{Finite time behavior for the parabolic problem} 
In this section we describe the behavior of the solution $u_\e$ to \eqref{prob} for large (but finite) time.  We state and prove the following theorem: for technical reasons we restrict our analysis to the case $f(u)=u^2/2$. We consider at first the (most) interesting case  of $u_0$ of type B; other types of initial data are described below after the proof.
\begin{theorem}\label{teo:finitetime}
Let $u_\e$ be the solution to \eqref{prob} with initial datum $u_0$ of type B. Then for any $\gamma, \delta>0$ there exist a time  $T=T_{\gamma,\delta}$ and $\bar \e=\bar \e_T$ such that for any $\e \leq \bar \e$ it holds that
\begin{equation}\label{tesi}
\begin{aligned}
& u_\e(x,T) < x-x_0+\delta \quad \mbox{for} \quad x \in  [\gamma,\ell] \\
& u_\e(x,T) > x-x_0-\delta \quad \mbox{for} \quad x \in  [0, \ell-\gamma]. \\
\end{aligned}
\end{equation}
\end{theorem}

\begin{proof}
Without loss of generality, we can suppose $u_0$ to be {\it monotone increasing} near $x=x_0$. Given $\gamma$ and $\delta$ some small fixed numbers, by Proposition \ref{prop14} we know that there exists $T_1$ such that 
$$\lim_{\e \to 0} u_\e(x,t) = U(x,t)$$
uniformly on any compact set of $[\gamma, \ell-\gamma] \times [T_1, \infty)$. Going further, from Theorem \ref{th:asyhyp} there exists a time $T_2 \geq T_1$ such that
$$|U(x,t)-(x-x_0)| < \frac{\delta}{2}, \qquad \mbox{for} \  (x,t) \in [\gamma,\ell-\gamma] \times [T_2, +\infty].$$
Hence, for any $T \geq T_2$ there exists $\e_0$ depending on $T$ such that for $\e \leq \e_0$
$$|u_\e(x,T)-(x-x_0)| < \delta, \qquad \mbox{for} \  x \in [\gamma,\ell-\gamma].$$
To complete the proof, we now need to describe the behavior of $u_\e$ near the boundary, that is we consider $x \in [0,\gamma] \cup [\ell-\gamma,\ell]$. To this purpose, after fixing $\alpha>1$, we define, for $x \in (0,x_0)$, the function
$$V(x,t)= \frac{\alpha(x-x_0-\delta/2) \, e^t}{\alpha e^t+1-\alpha},$$
to be used as a barrier function. There holds that (notice that $V_{xx}=0$ and recall that $f'(V) = V$)
\begin{equation*}
\begin{aligned}
&V_t -\e \left( h(V) V_x \right)_x+ V\left(V_x-1 \right) = \\
& = \frac{ \alpha(x-x_0-\frac{\delta}{2}) \, e^t \, (1-\alpha)}{(\alpha e^t+1-\alpha)^2} -\e h'(V)V_x^2  + \frac{\alpha (x-x_0-\delta/2)  \, e^t}{\alpha e^t+1-\alpha}\cdot \frac{(\alpha-1)}{\alpha e^t+1-\alpha} \\
& = -\e h'(V)V_x^2.
\end{aligned}
\end{equation*}
The resulting term is negative since $h'(V)  \geq 0$ by \eqref{segnodih}, being $V(x,t) <0$ for $x \in (0,x_0)$. Hence, the function $V(x,t)$ is a sub-solution to \eqref{prob}. Moreover
\begin{equation*}
V(0,t)<0=u_\e(0,t) \quad \mbox{and} \quad V(x_0,t)=-\frac{\alpha \, \delta/2 \, e^t}{\alpha e^t+1-\alpha} <-\frac{\delta}{2},
\end{equation*}
while, if $\alpha$ is large enough
$$
V(x,0)=\alpha(x-x_0-\delta/2) < u_0(x).
$$
Finally, for $t \to +\infty$, the function $V(x,t) \to x-x_0-\delta/2$, meaning that there exists $T_3 \geq T_2$ sufficiently large so that one has 
$$V({t},x) > x-x_0-\delta,$$
for all $x \in (0,\ell)$ and $t \geq T_3$. Let now $\e_1$ be such that $u_\e(x_0,t) >-\frac{\delta}{2}$ for $t \in [0,T_3]$ and $\e \leq \e_1$; hence, by the comparison principle (and recalling that $V$ is sub-solution for the problem solved by $u_\e$) we have
$$u_\e(x,T_3) \geq V(x,T_3) > x-x_0-\delta, \qquad \mbox{for} \  x \in [0,x_0].$$
The second inequality of \eqref{tesi} then follows by choosing $T_{\gamma,\delta}=T_3$ and $\bar \e= \min\{\e_0,\e_1 \}$. In order to prove the first inequality one can proceed as above by using
$$V_1(x,t)= \frac{\beta(x-x_0+\delta/2) \, e^t}{\beta e^t+1-\beta},$$
 as a super-solution to \eqref{prob} in $(x_0,\ell)$, for $\beta$ large enough.
\end{proof}
Combining estimates \eqref{tesi}, one can deduce that for any (arbitrarily small) $\gamma, \delta>0$ there exist a time  $T=T_{\gamma,\delta}$ and $\bar \e$ (depending on $T$) such that for any $\e \leq \bar \e$
$$|u_\e(x,T)-(x-x_0)| < \delta, \qquad \mbox{for all} \  x \in [\gamma, \ell-\gamma].$$
The latter estimate shows that the solution to \eqref{prob} with initial datum $u_0$ of type B  becomes close to the line $x-x_0$ for $t=T$, as   {depicted in Figure \ref{fig1}}, for $T=5$. 
However, once the solution has this shape, it maintains it for a long time interval (in Figure \ref{fig1} the solution is still there for $T_1= \mathcal{O}(10^8)$). Hence, even if $u_{1^-,\e}$ is not a stable configuration for the system (recall Theorem \ref{stabthm} and the subsequent remark) the motion of the solution towards one of the equilibrium configurations occurs in an exponentially slow way (for more details, see the numerical analysis performed in the following Section).

\begin{figure}[hbtp]
\centering
\includegraphics[scale=0.4]{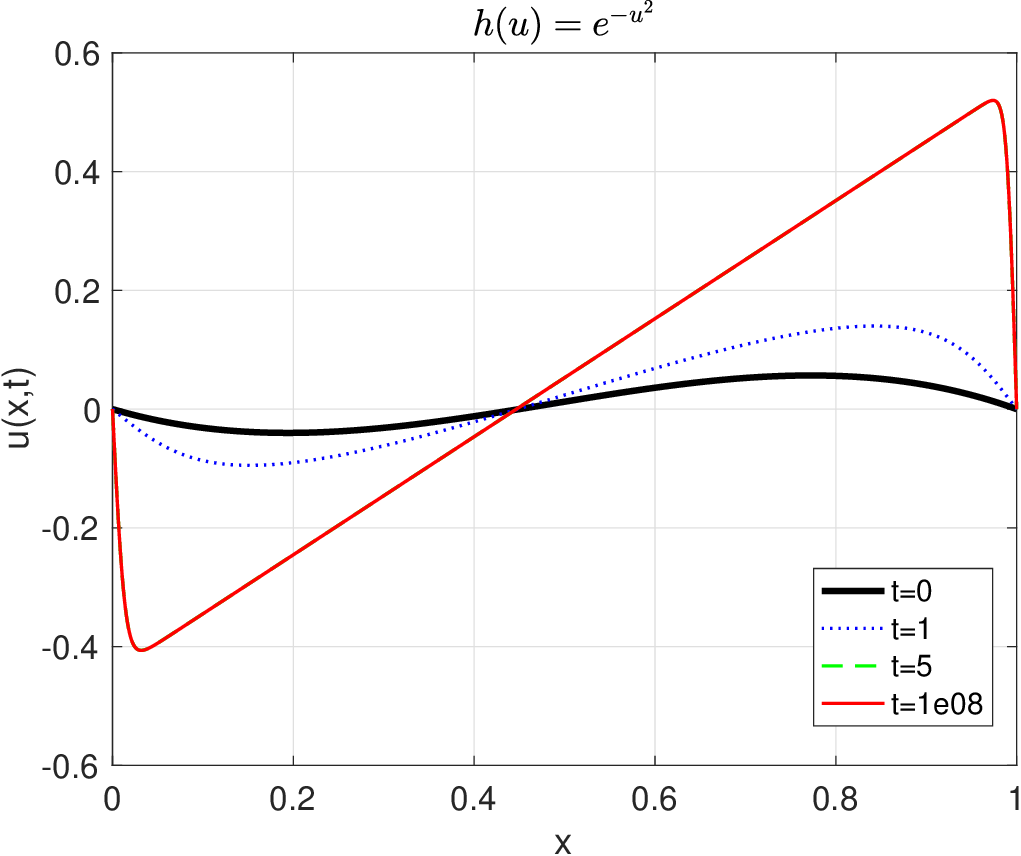}
\caption{\small{Here $h(u) = e^{-u^2}$, $\ell=1$, $\varepsilon = 0.003$ and the initial datum is $u_0(x)= x(1-x)(x-0.45)$; the solution becomes close to the line $x-x_0$, with $x_0=0.45$, for $T=5$ and remains almost still at least until $T_1=10^8$.}}
\label{fig1}
\end{figure}

\subsubsection*{Comments on other types if initial data}
Similarly to Theorem \ref{teo:finitetime}, one can prove the following results.
\begin{itemize}
\item If $u_0$ is of type A, then for any $\gamma, \delta>0$ there exist $T_1$ and $\bar \e_1$ such that, for any $\e \leq \bar \e_1$, we have 
$$\mbox{either}\quad |u_\e(x,T_1)-x| < \delta \qquad \mbox{or} \quad |u_\e(x,T_1)-(x-\ell)| < \delta,$$
for all $x \in [\gamma, \ell-\gamma]$, depending on whether $u_0>0$ or $u_0<0$.
\item If $u_0$ is of type C, then for any $\gamma, \delta>0$ there exist $T_2$ and $\bar \e_2$ such that, for any $\e \leq \bar \e_2$ 
$$|u_\e(x,T_2)-\left(x \, \chi_{\left[0,x_0 \right)} + (x-\ell) \chi_{\left(x_0, \, \ell \right]}\right)| < \delta,$$
for all $x \in [\gamma, \ell-\gamma]$.
\end{itemize}

The aforementioned statements are depicted in Figures \ref{fig2}-\ref{fig3}. We notice that if $u_0$ does not change sign the corresponding time dependent solution becomes close to either $x$ or $x-1$ in short times (in both cases $T_1=2.5$, see Figure \ref{fig2}). 
\begin{figure}[hbtp]
\includegraphics[scale=0.4]{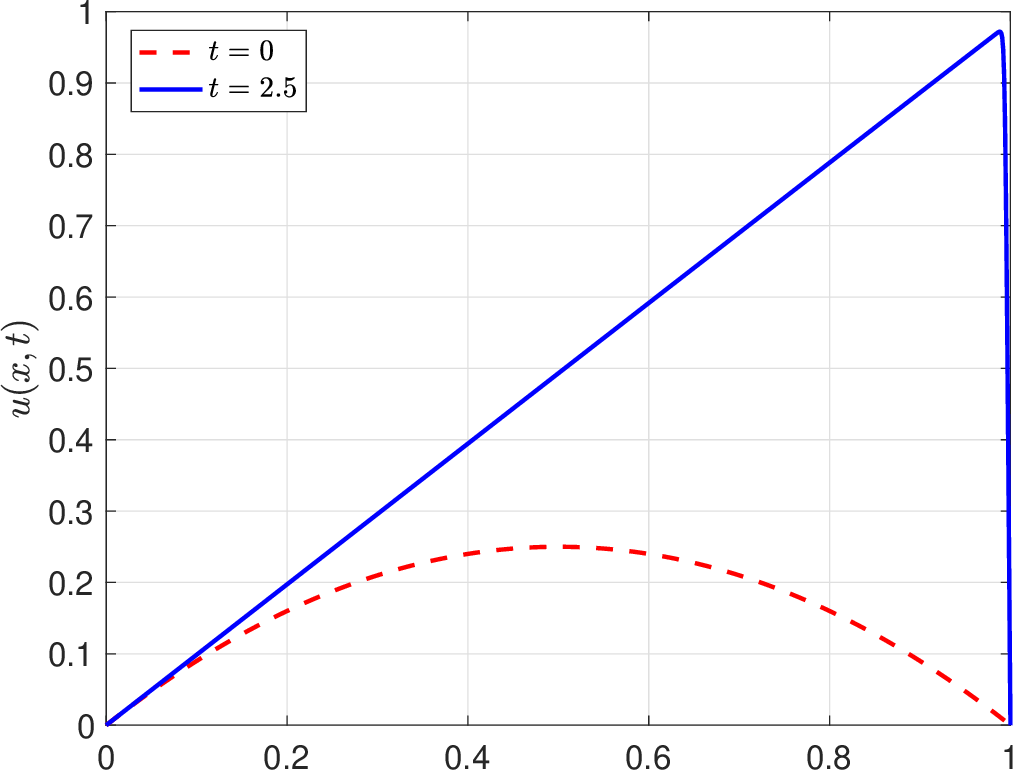}
\includegraphics[scale=0.4]{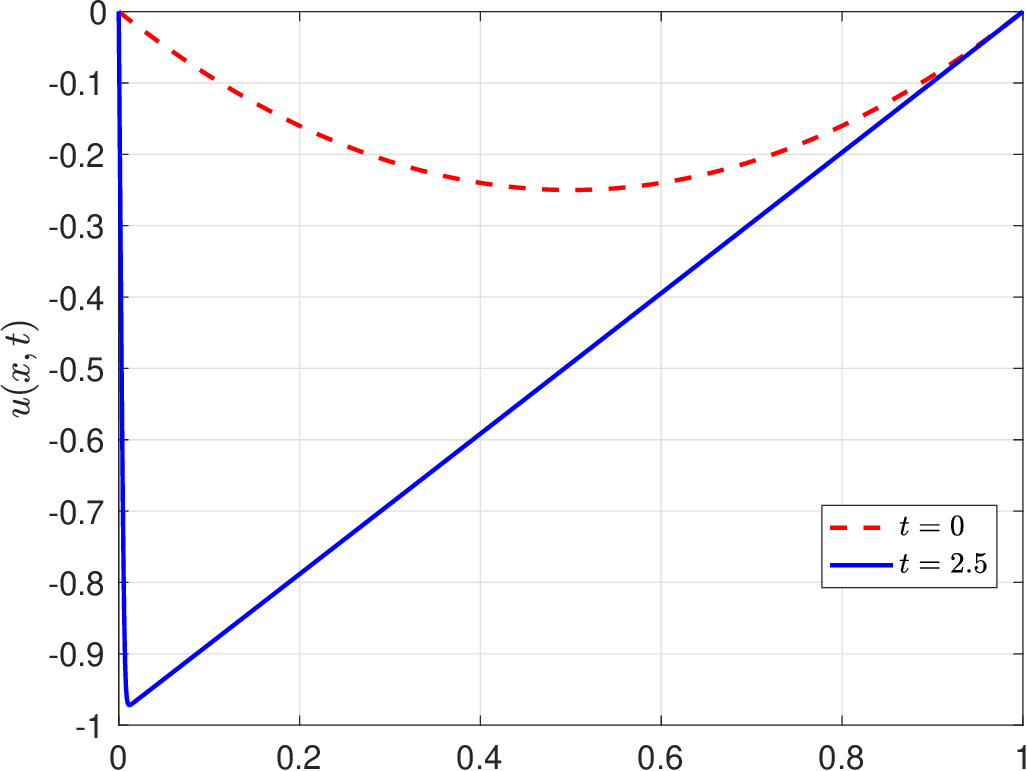}\\
\caption{\small{Here $h(u) = e^{-u^2}$, $\ell=1$, $\varepsilon = 0.003$ and the initial data are (from left to right respectively) $u_0=x(1-x)$ and $u_0= -x(1-x)$. }}
\label{fig2}
\end{figure}

Similarly,  in the case of $u_0$ of type C (Figure \ref{fig3}), the solution becomes rapidly close to the function  $x \, \chi_{\left[0,x_0 \right)} + (x-\ell) \chi_{\left(x_0, \, \ell \right]}$ in a finite (short) time $T_2=2.5$.
\begin{figure}[hbtp]
\begin{center}
\includegraphics[scale=0.4]{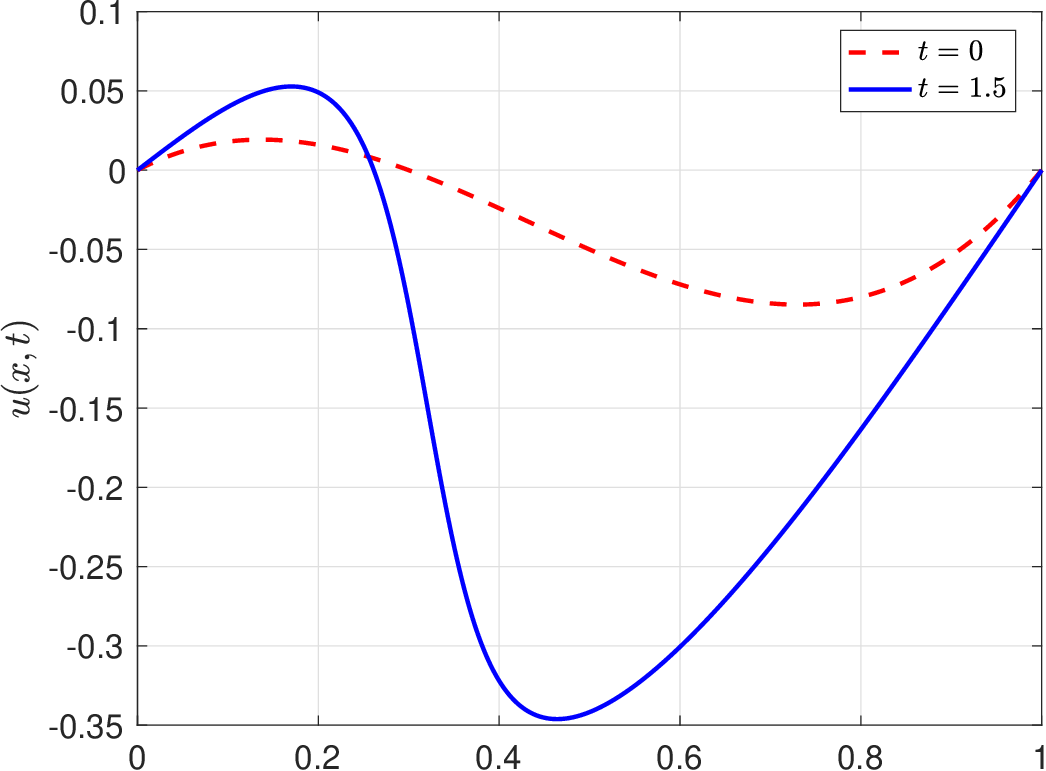}
\includegraphics[scale=0.4]{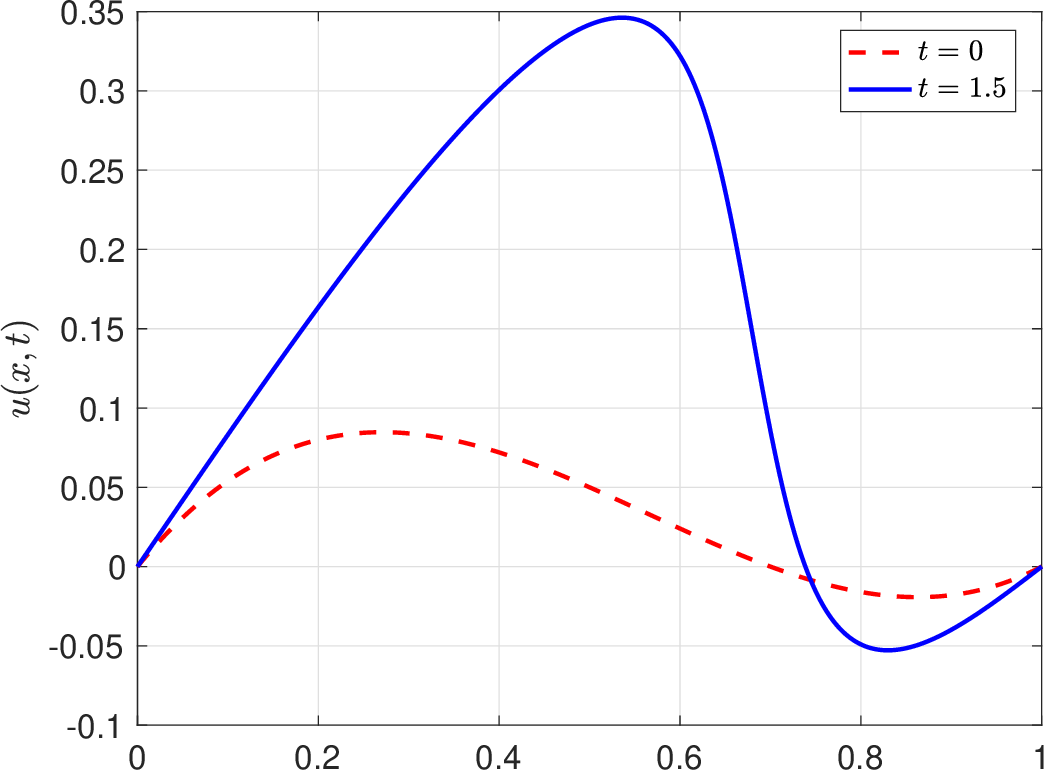}
\caption{\small{Here $h(u) = e^{-u^2}$, $\ell=1$, $\varepsilon = 0.003$. The initial data are $u_0=-x(1-x)(x-0.3)$ on the left figure and $u_0=-x(1-x)(x-0.7)$ on the right. }}
\label{fig3}
\end{center}
\end{figure}

\section{Numerical solutions}\label{numerics}
We conclude the paper with some numerical simulations showing the evolution of the solutions to \eqref{prob}. All the numerical computations are done with the purpose of illustrating the theoretical results concerning the stability properties of the steady states $u_{\pm,\e}$ and the instability ones of the states ${u_{1^\pm,\e}}$. 

Numerical computations are performed using the built-in solver \texttt{pdepe} by \textsc{Matlab}$^\copyright$, which is a set of tools to solve PDEs in one space dimension.

In all the examples we consider problem \eqref{prob} with $h$ given by either $h(u)=e^{-u^2}$ or $h(u)= \frac{1}{1+u^2}$, while $f(u)=\frac{u^2}{2}$; all these functions satisfy assumptions \eqref{segnodih}-\eqref{segnodif}.

\subsection{Test 1: fast convergence} 
We start with an example illustrating the stable behavior of the steady states $u_{\pm,\e}$, hence illustrating the results of Theorem \ref{stabthm}. In the first numerical simulation we consider the function $h(u)=e^{-u^2}$ and we choose initial data that do not change sign in the interval $(0,\ell), \ell=1$. Precisely, on the right picture we have $u^+_0=x(1-x)$ (which is always positive), while on the left $u^-_0=-u_0^+= -x(1-x)$.
\begin{figure}[h]
\includegraphics[scale=0.4]{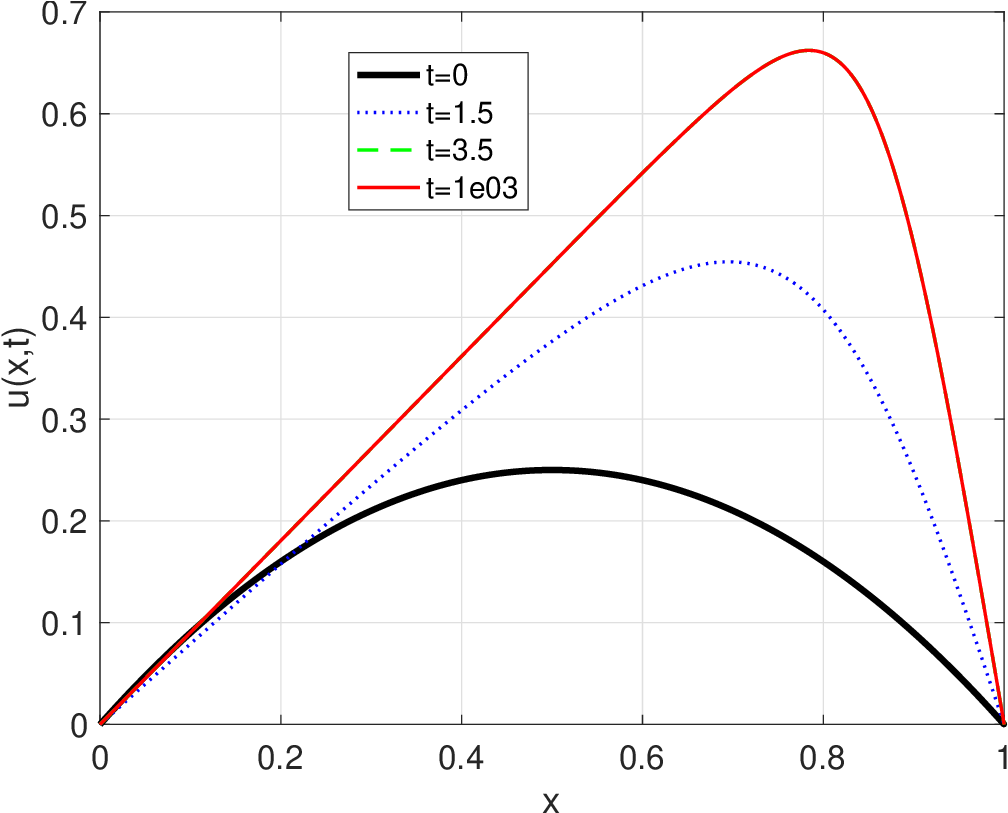}
\includegraphics[scale=0.4]{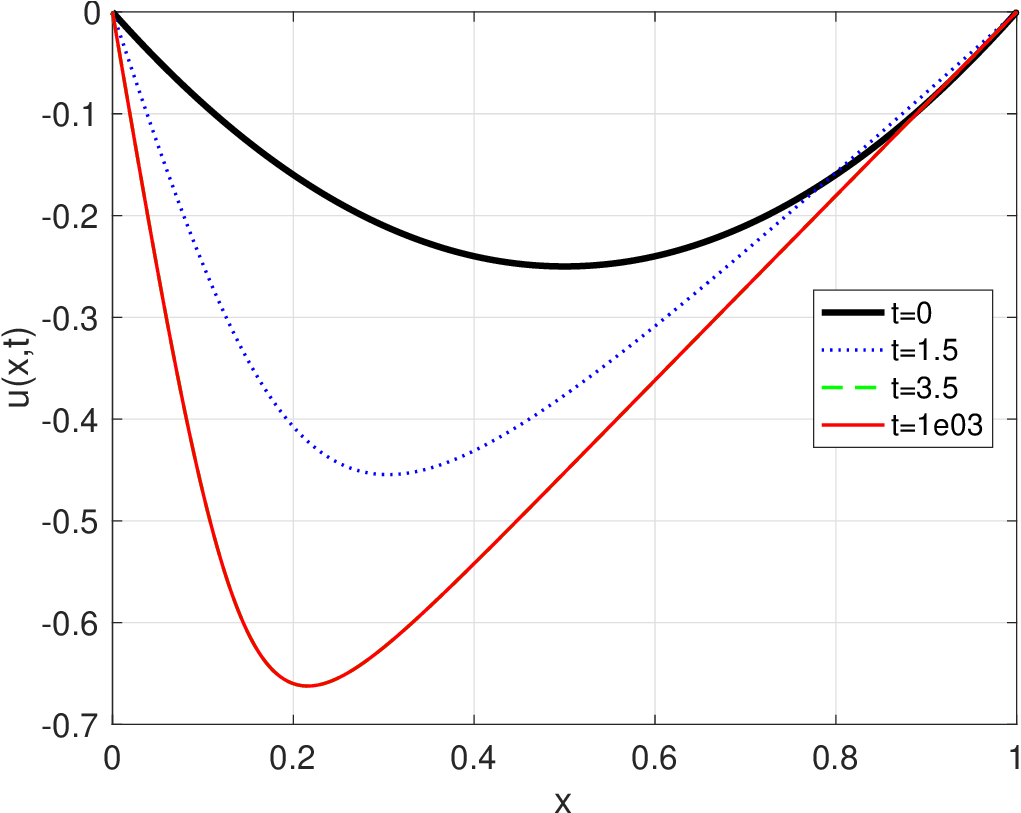}
\caption{\small{Test 1 - fast convergence to the stable steady states $u_{\pm,\e}$ for $\e=0.06$. The solution has already reached the stable profile at $t=3.5$; indeed, the solutions at $t=3.5$ and $t=10^3$ are indistinguishable.} }
\label{figB}
\end{figure}

In Figure \ref{figB}, we plot the solution of the IBVP \eqref{prob} for different times $t$ and $\e=0.06$.
 It can be observed that the solution rapidly converges to either $u_{+,\e}$ (on the left) or $u_{-,\e}$ (on the right). More precisely, the
solution at time $t \approx 3.5$ has basically already reached its final equilibrium configuration, and for $t>3.5$  it is still.
Similarly, in Figure \ref{figA}, we consider the same data but with $\varepsilon = 0.006$. We can notice how the dynamics remains unchanged with respect to the parameter $\e$; precisely, we have still convergence towards either  $u_{+,\e}$ or $u_{-,\e}$ for $t\approx 3.5$ (as before), even if $\e$ is smaller.

\begin{figure}[hbtp]
\includegraphics[scale=0.4]{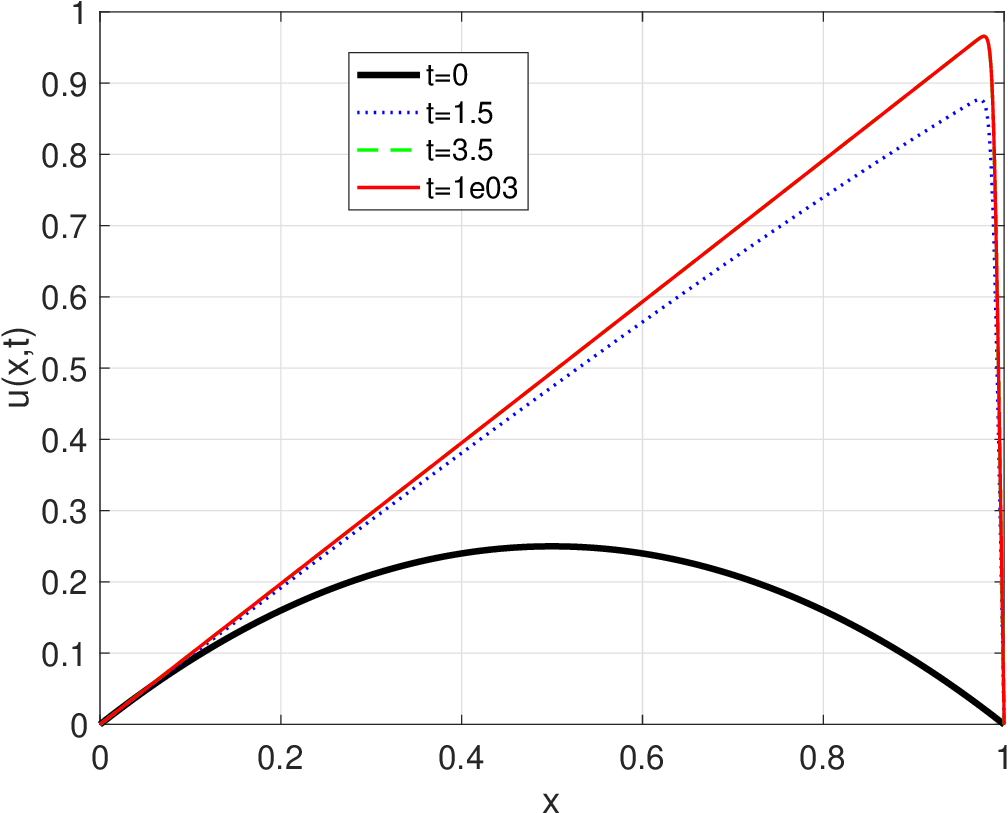}
\includegraphics[scale=0.4]{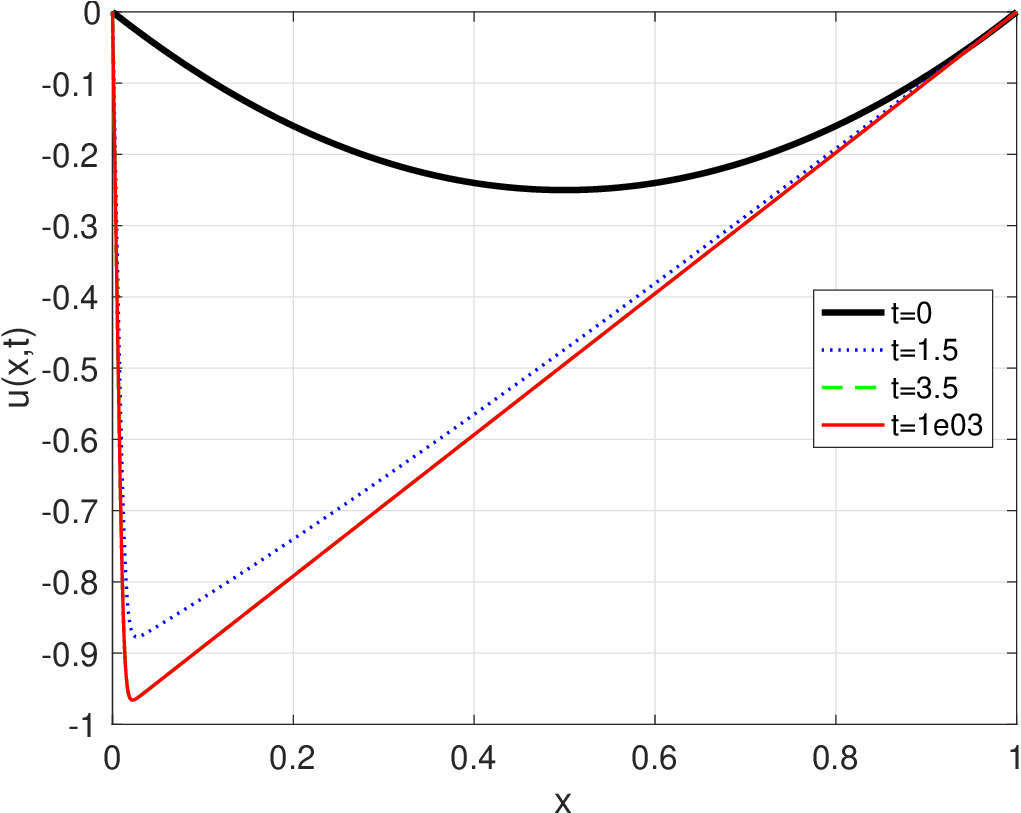}
\caption{\small{Test 1 - fast convergence to the stable steady states $u_{\pm,\e}$ for $\e=0.006$. The solutions at times $t=3.5$ and $t=10^3$ are indistinguishable. }}
\label{figA}
\end{figure}

\subsection{Test 2: instability of the steady state} 
In the next numerical simulation, depicted in Figure \ref{figC}, we show the unstable nature of the steady state  $u_{1^+,\e}$ (see Remark \ref{INSTAB}). In this case we fix $\e=0.06$ and we consider $h(u)=e^{-u^2}$ on the left picture and $h(u)=(1+u^2)^{-1}$ on the right one. The initial datum is given by $u_0(x)=-x(1-x)(x-0.45)$ in the first top panels and $u_0(x)=-x(1-x)(x-0.55)$ in the bottom ones.

\begin{figure}[hbtp]
\includegraphics[scale=0.4]{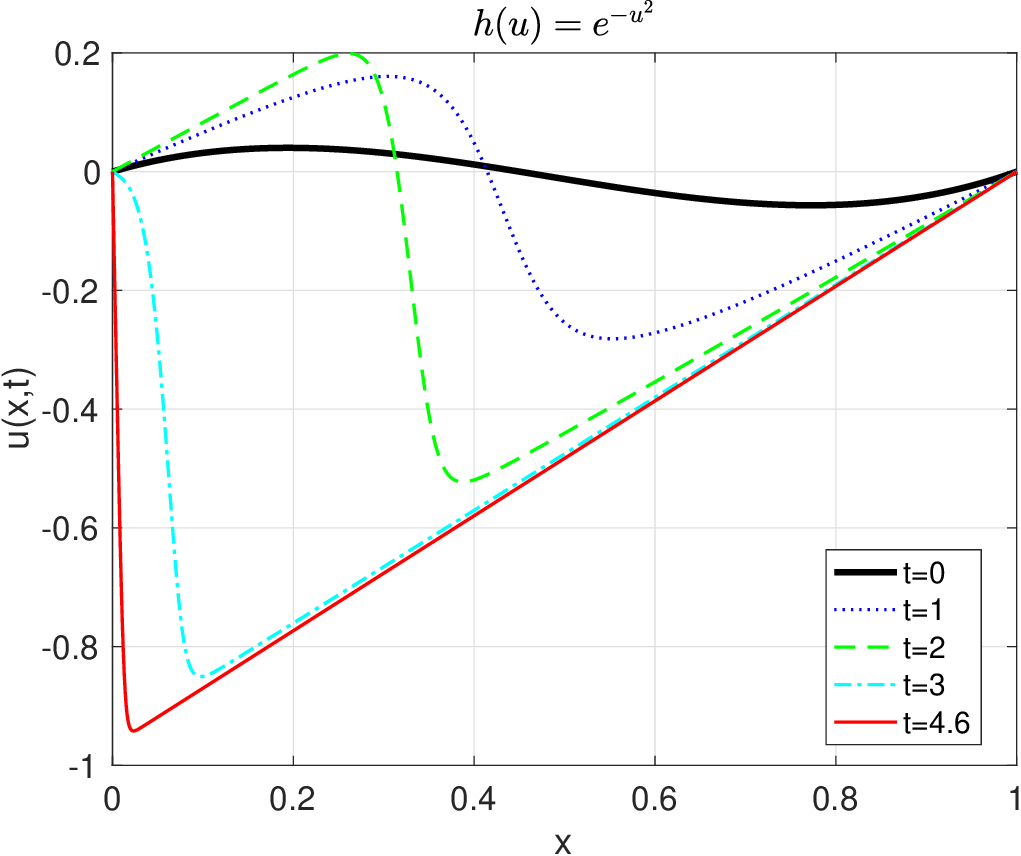}
\includegraphics[scale=0.4]{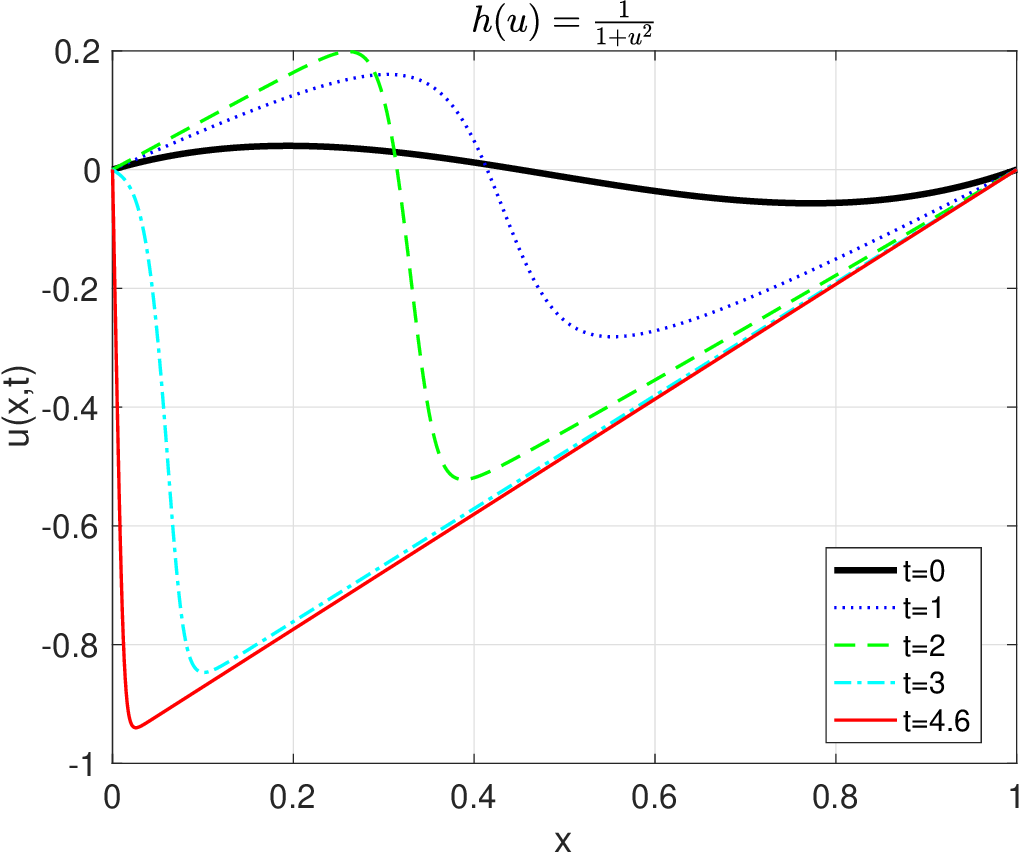}\\
\includegraphics[scale=0.4]{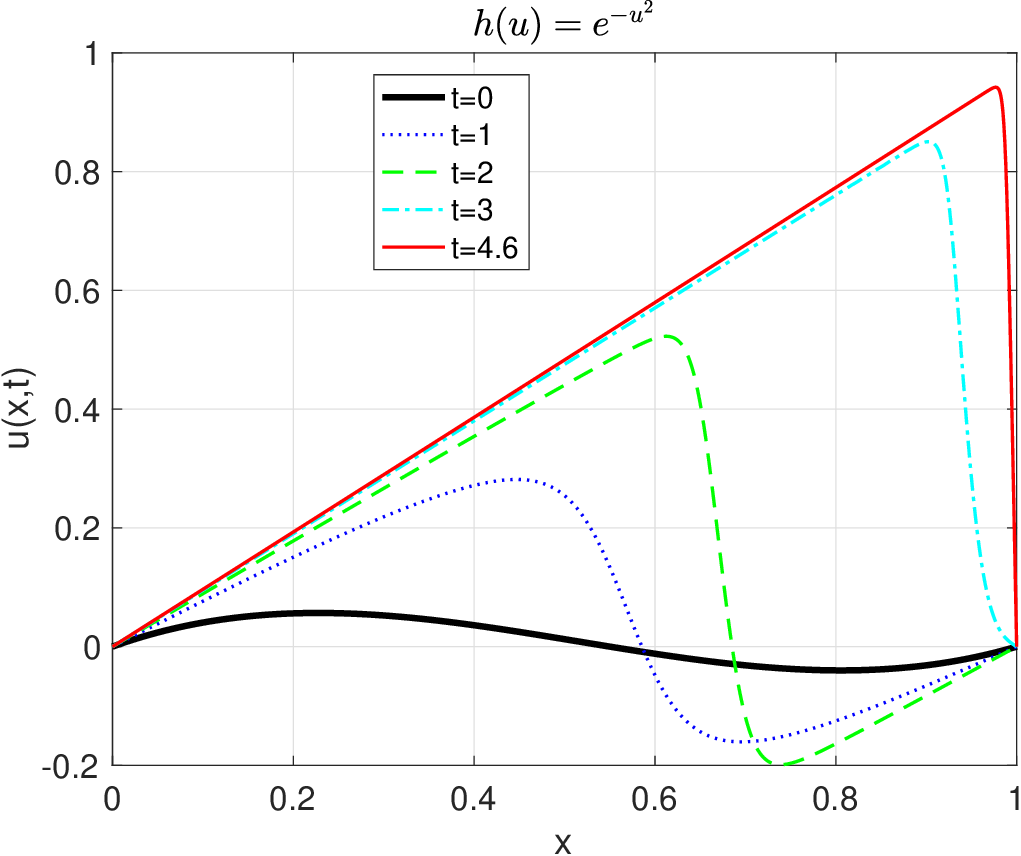}
\includegraphics[scale=0.4]{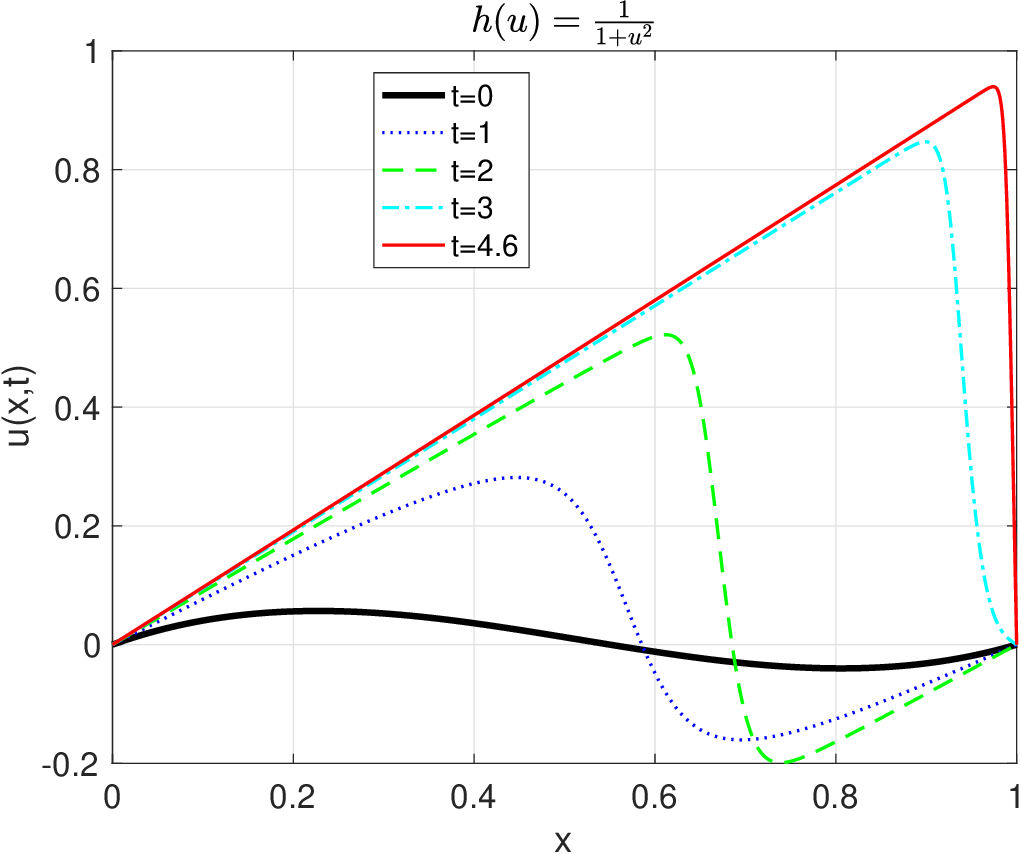}\\
\caption{\small{Test 2 - fast convergence to the stable steady states $u_{\pm,\e}$ for $\e=0.006$. The initial datum  changes sign once in $(0,1)$ and it is given by either $u_0=-x(1-x)(x-0.45)$ (upper figures)  or $u_0=-x(1-x)(x-0.55)$ (lower figures).  } } 
\label{figC}
\end{figure}
We see that the solution is close to $u_{1^+,\e}$ for short times;
however, since $u_{1^+,\e}$ is not a stable steady state, we have a  subsequent (fast) motion towards either $u_{-,\e}$ (see the plots in the first row, where $x_0 =0.45< 1/2$) or $u_{+,\e}$ (plots in the second row, where $x_0=0.55$), and this motion occurs again in short times.

We finally notice how the evolution of the solution remains basically unchanged with respect to the choice of the function $h$.

\subsection{Test 3: metastable behavior} 
In this numerical test, we  consider initial data of type B\footnote{
$u_0 <0$ in $(0,x_0)$ and $u_0>0$ in $(x_0,\ell)$ for some $x_0 \in (0,\ell)$} for which the so called {\it metastable behavior} appears. In particular, we see that the solution exhibits a first transient phase in which it is
close to some non-stable stationary state (in this case $u_{1^-,\e}$) for an exponentially long time before converging to its asymptotic limit.

\begin{figure}[hbtp]
\includegraphics[scale=0.4]{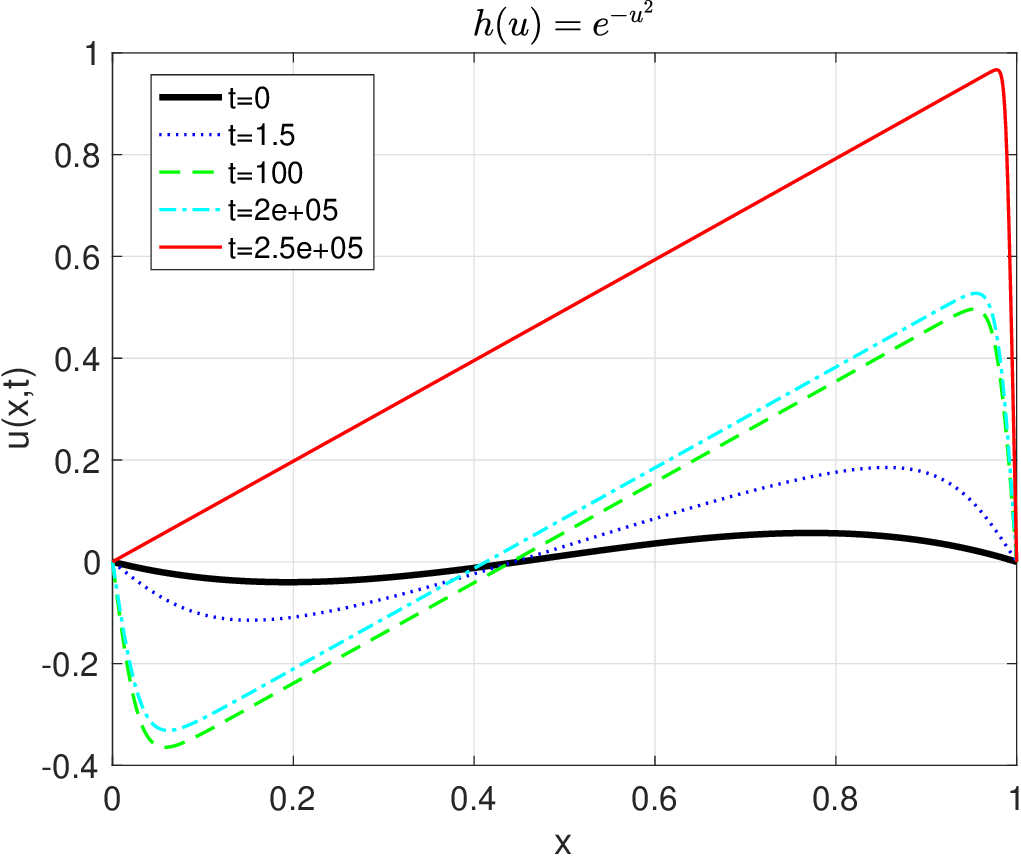}
\includegraphics[scale=0.4]{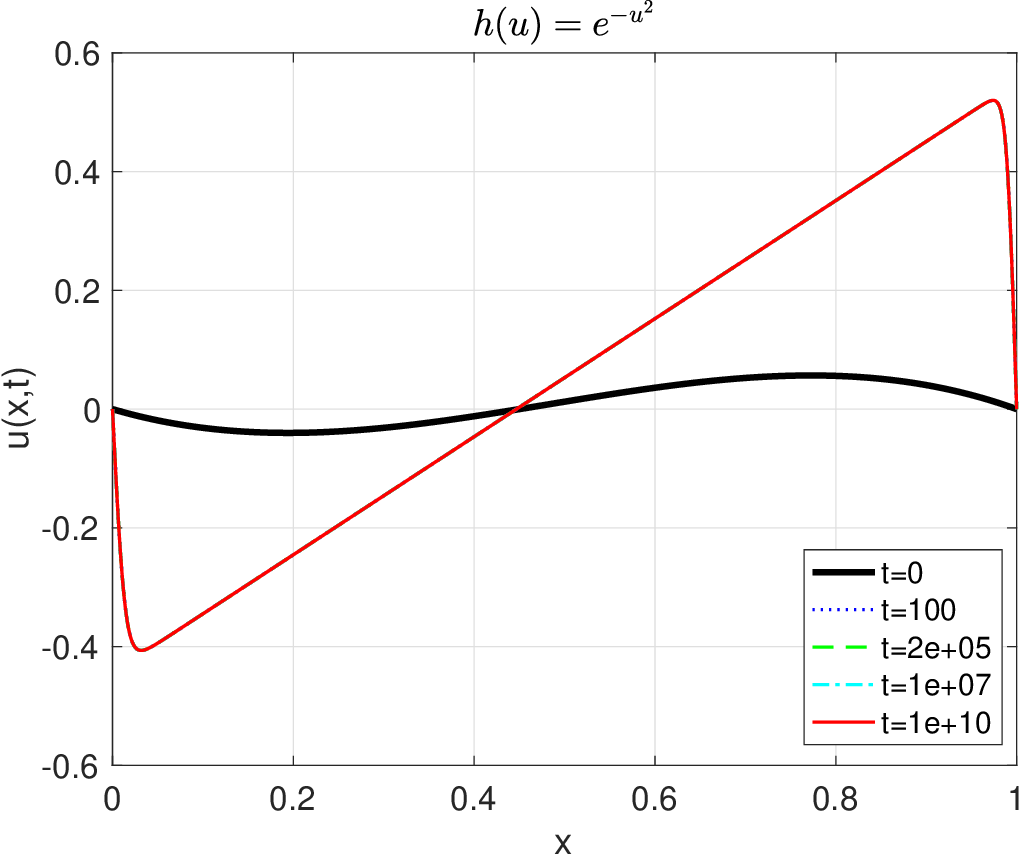}\\
\includegraphics[scale=0.4]{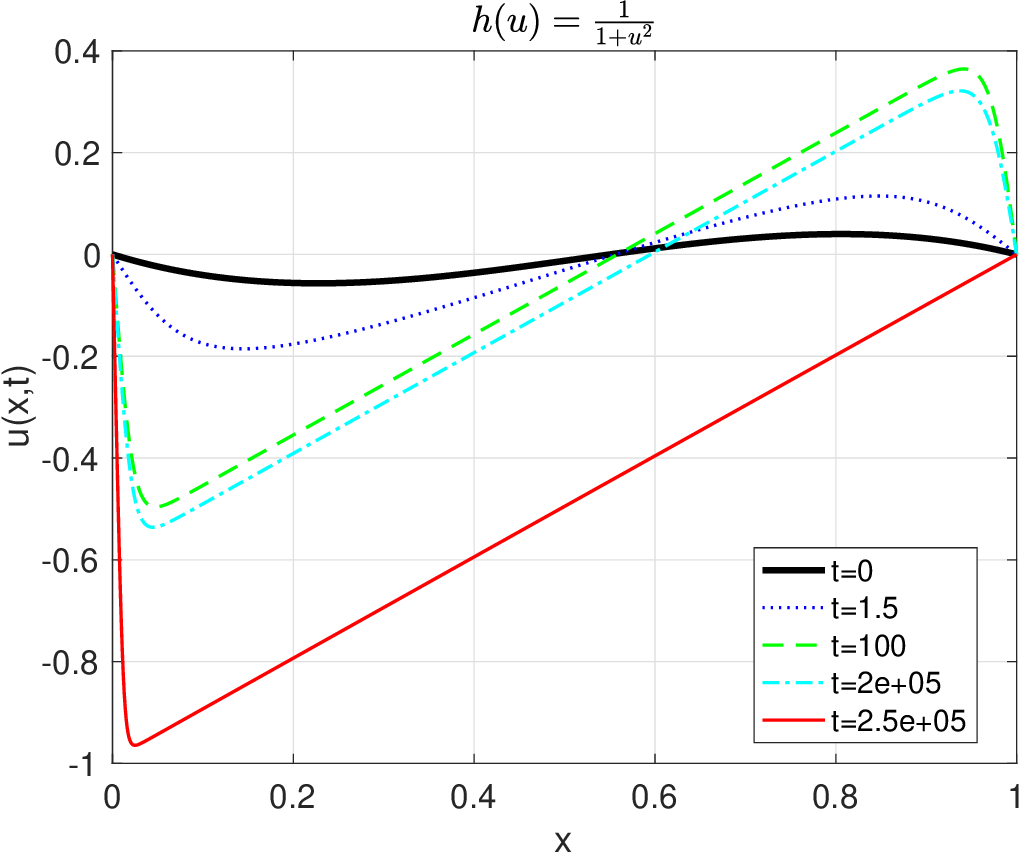}
\includegraphics[scale=0.4]{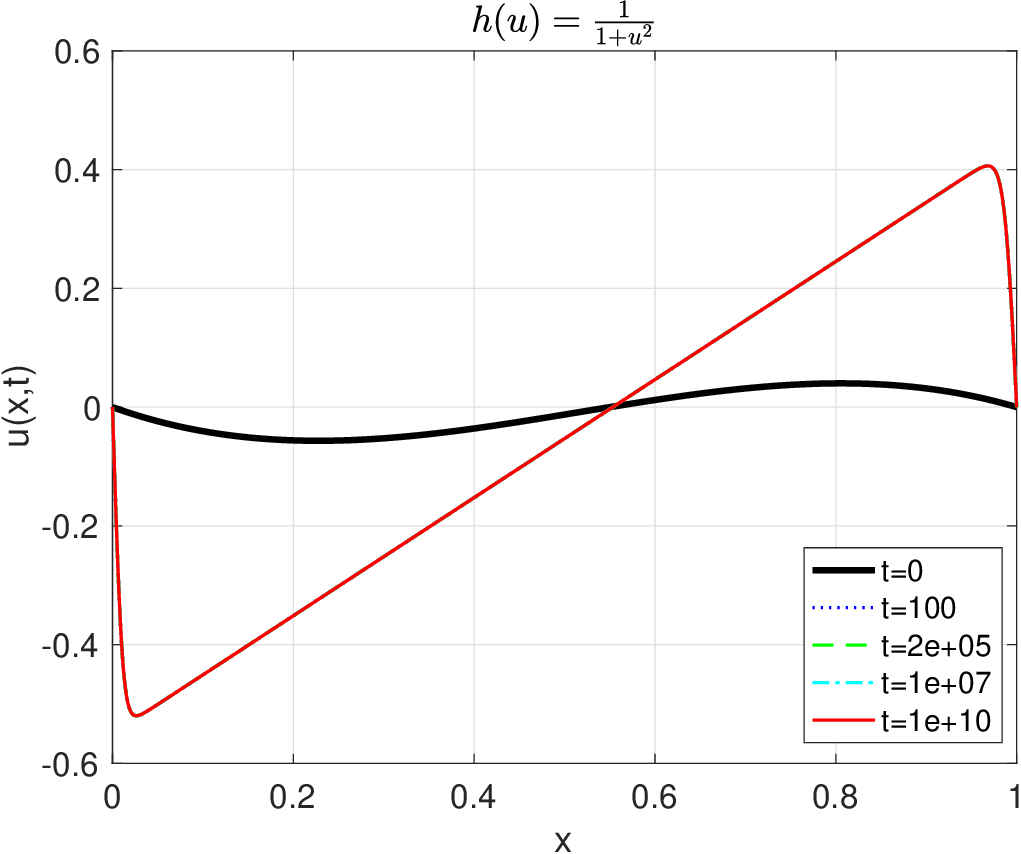}
\caption{\small{Test 3 - we choose $\varepsilon = 0.006$ on the left picture and $\e=0.003$ on the right. }}
\label{figD}
\end{figure}

In the top panels of Figure \ref{figD}, we choose:
$$h(u)=e^{-u^2}, \quad f(u)=\frac{u^2}{2} \qquad \mbox{and} \qquad u_0(x)=x(1-x)(x-0.45),$$
whereas in the bottom panels we study $h(u)=(1+u^2)^{-1}.$
In the top-left panel of Figure \ref{figD}, we can see that the solution develops a profile close to $u_{1^-,\e}$ in a relatively short time, and maintains this configuration up to $t= 2\times10^5$; hence, even if $u_{1^-,\e}$ is an unstable configuration for the system, the motion of the solution towards one of the equilibrium solutions occurs in an exponentially slow way. Precisely, the (meta)stable state $u_{1^-,\e}$ appears stable until times of the order $\mathcal{O}(10^5)$ and, in this case, we observed numerical convergence towards $u_{+,\e}$ for $t=2.5\times 10^5$.

In the top-right panel of Figure \ref{figD}, we decrease $\e= 0.003$ and we see that the solution is  still even for $t=10^{10}$. The numerical simulations was stopped for memory issues.
This result is in line with the theory developed: in particular, the smaller $\e$ the larger the time to reach the stable configuration.
Similar considerations can be done for the plots in the second row, where we only changed the function $h(u)$. In the pictures we kept the same time behavior. 
In addition, a numerical study for the time $T$ needed to reach the stable solution using different functions $h$ is shown in Table \ref{tab:meta}. As expected, we can see that if we decrease $\e$, the time needed increases and that if $h(u) = u$ the convergence happens to be faster. In the case  $h(u)=({1+u^2})^{-1}$ the time needed  to reach the stable solution is slightly smaller than the case with $h(u)=e^{-u^2}$. The results are clearly comparable.

\begin{table}[htbp]
\centering
\begin{tabular}{|c|c|}
\hline 
$\varepsilon$ & T\\
\hline
0.024 & 4\\
0.012 & 20\\
0.006 & 7.7e+04\\
\hline
\end{tabular}
\begin{tabular}{|c|c|}
\hline 
$\varepsilon$ & T\\
\hline
0.024 & 3.2\\
0.012 & 21.75\\
0.006 & 2.14e+05\\
\hline
\end{tabular}
\begin{tabular}{|c|c|}
\hline 
$\varepsilon$ & T\\
\hline
0.024 & 6\\
0.012 & 22\\
0.006 & 2.5e+05\\
\hline
\end{tabular}
\vspace{0.5cm}
\caption{Test 3 - time needed to reach the stable solution for different $\e$. The three tables consider $h(u) = 1$ on the left, $h(u)=\frac{1}{1+u^2}$ in the middle and $h(u)=e^{-u^2}$ on the right.}
\label{tab:meta}
\end{table}

Finally,  in Figure \ref{fig:spost} we show how the intersection with the $x-$axis of the solutions of \eqref{prob} changes  over time. We tested different $\e$ and one can see the speed of the convergence, according to what reported in Table \ref{tab:meta}. The study is done for $h(u)= e^{-u^2}$.

\begin{figure}[htbp]
\centering
\includegraphics[scale=0.4]{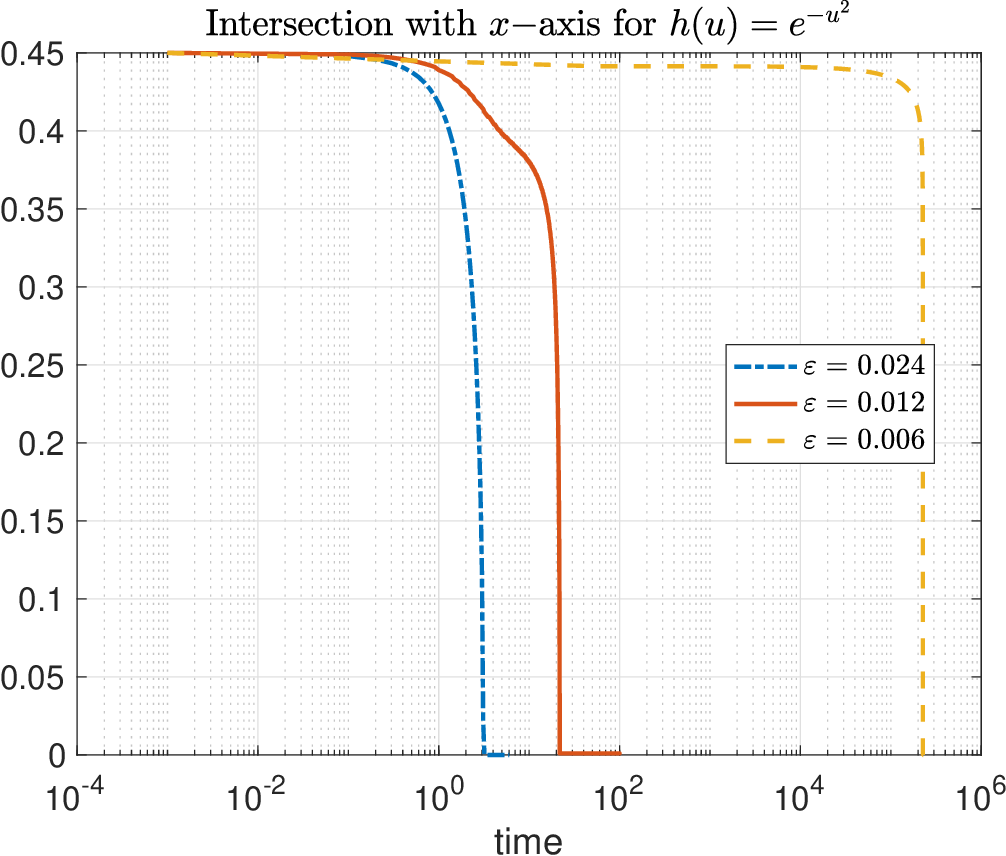}
\caption{Test 3 - intersection of $u(x,t)$ with the $x-$axis for different $\e$.}
\label{fig:spost}
\end{figure}

\newpage
\subsection{Test 4: Discontinuous initial data} 
These tests illustrate how the metastable behavior is still present even if considering  either a discontinuous initial condition (see \eqref{ex51}) or an initial datum which is continuous but with discontinuous derivative at some points of the interval (see \eqref{ex52}).
\begin{equation}\label{ex51}
u_0(x) = 
\left\{\begin{aligned} 
&- x \quad &\mbox{if} \quad x \in [0,0.48],\\
&0.5-\frac{0.5}{0.7} (x-0.3) \quad &\mbox{if} \quad x \in [0.48,1].
\end{aligned}\right.
\end{equation}
\begin{equation}\label{ex52}
u_0(x) = 
\left\{\begin{aligned} 
&- 0.5 x \quad &\mbox{if} \quad x \in [0,0.2],\\
&-0.1+\frac{0.2}{0.6}( x-0.2) \quad &\mbox{if} \quad x \in [0.2,0.8],\\
&- 0.5 (x-1) \quad &\mbox{if} \quad x \in [0.8, 1].\\
\end{aligned}\right.
\end{equation}
In both cases we can see that the solution evolves into a continuous profile in a relatively short time; we also observe that the zero of the solution at time $t=1.5$ is located exactly at the discontinuous point (left picture) or at the zero of the initial datum (right picture). After developing a smooth profile, the solutions will exhibit the same metastable behavior of the previous examples and slowly converge to one of the stable steady states of the system.

\begin{figure}[hbtp]
\includegraphics[scale=0.4]{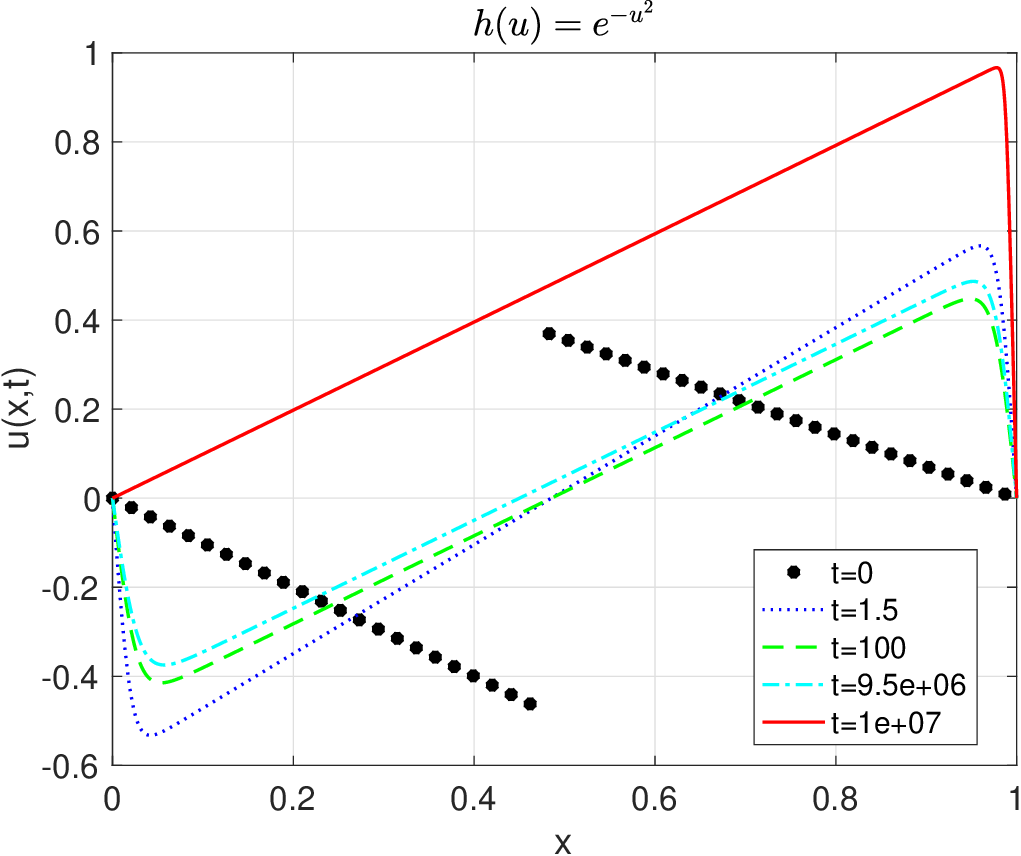}
\includegraphics[scale=0.4]{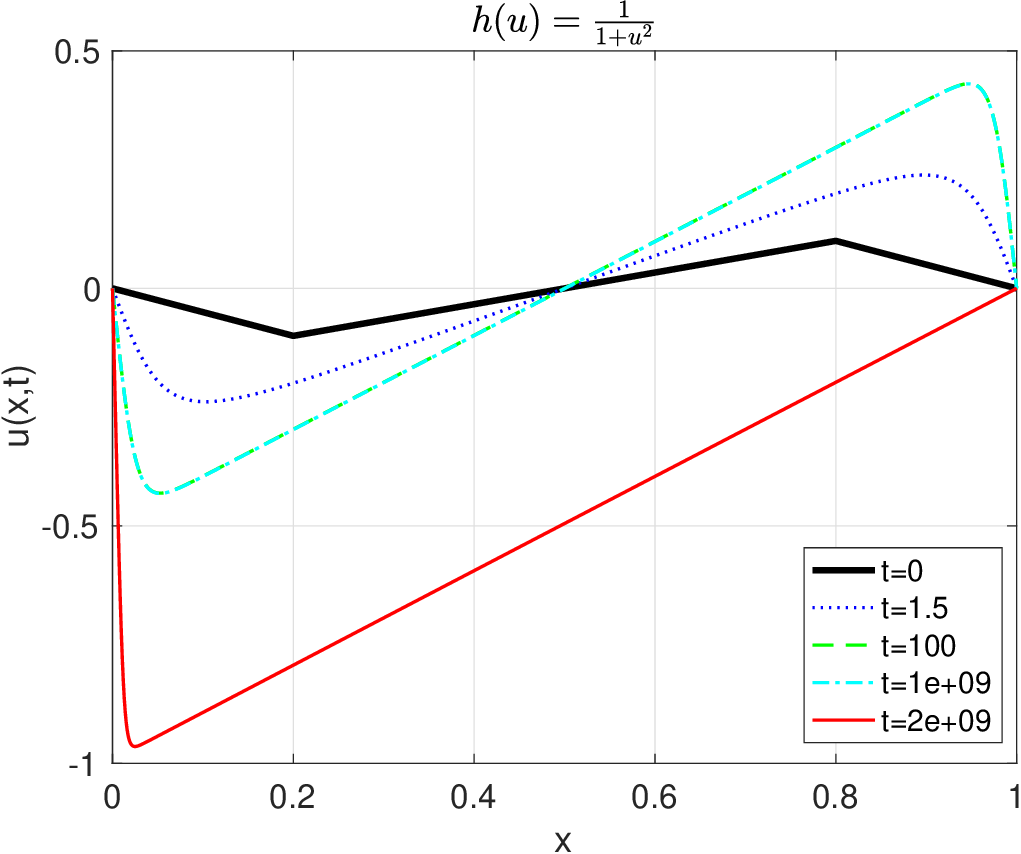}
\caption{Test 4: We choose $\varepsilon = 0.006$ for $u_0$ given in \eqref{ex51} on the left picture and $u_0$ in \eqref{ex52} on the right.}
\end{figure}

\section*{Acknowledgments} 
A. Alla has developed this work within the activities of the project ``Data-driven discovery and control of multi-scale interacting artificial agent systems” (code P2022JC95T), funded by the European Union -- NextGenerationEU, National Recovery and Resilience Plan (PNRR) – Mission 4 component 2, investment 1.1 ``Fondo per il Programma Nazionale di Ricerca e Progetti di Rilevante Interesse Nazionale" (PRIN).
A.A. is also supported by MIUR with PRIN project 2022 funds (P2022238YY5, entitled ‘‘Optimal control problems: analysis, approximation’’). The work of A. De Luca  is partially supported by the INdAM-GNAMPA 2025 grant ``PDE ellittiche che degenerano su variet\`{a} di dimensione bassa e frontiere libere molto sottili'' CUPE5324001950001. 
The work of R. Folino is partially supported by DGAPA-UNAM, program PAPIIT, grant IN-103425.

A. Alla is a member of GNCS INDAM research group.  A. De Luca and M. Strani are members of GNAMPA INDAM research group.

\end{document}